\documentclass{amsart}
\usepackage{amsmath}
\usepackage{amssymb}
\usepackage{amsthm}
\usepackage{enumerate}
\usepackage{tikz}
\usetikzlibrary{decorations.pathreplacing}
\usepackage[utf8]{inputenc}
\usepackage{etoolbox}

\makeatletter
\newtheorem*{rep@theorem}{\rep@title}
\newcommand{\newreptheorem}[2]{%
	\newenvironment{rep#1}[1]{%
		\def\rep@title{#2 \ref{##1}}%
		\begin{rep@theorem}}%
		{\end{rep@theorem}}}
\makeatother

\numberwithin{equation}{section}

\newtheorem{theorem}{Theorem}[section]
\newreptheorem{theorem}{Theorem}
\newtheorem{lemma}[theorem]{Lemma}
\newtheorem{proposition}[theorem]{Proposition}
\newreptheorem{proposition}{Proposition}
\newtheorem{corollary}[theorem]{Corollary}
\newreptheorem{corollary}{Corollary}

\theoremstyle{definition}
\newtheorem{example}[theorem]{Example}
\newtheorem{definition}[theorem]{Definition}
\newtheorem{construction}[theorem]{Construction}
\newtheorem{remark}[theorem]{Remark}

\newtheorem{notation}[theorem]{Notation}

\newtheorem*{remark*}{Remark}
\newtheorem{claim}{Claim}
\newtheorem*{claim*}{Claim}
\AtEndEnvironment{proof}{\setcounter{claim}{0}}
\newenvironment{claimproof}[1]{\par\noindent \underline{Proof of Claim\space#1}}{\par}

\newcommand{\sm}{\setminus}

\newcommand{\N}{\mathbb{N}}
\newcommand{\T}{\mathbb{T}}
\newcommand{\Z}{\mathbb{Z}}
\newcommand{\Q}{\mathbb{Q}}
\newcommand{\R}{\mathbb{R}}
\newcommand{\C}{\mathbb{C}}
\newcommand{\F}{\mathcal{F}}

\newcommand{\MM}{\mathcal{M}}
\newcommand{\LL}{\mathcal{L}}

\newcommand{\VV}{\mathcal{V}}

\newcommand{\saw}[1]{
	\left( \! \left(#1\right) \! \right)}
\newcommand{\saww}[1]{
	\Big( \!\! \Big(#1\Big)\!\!  \Big)}
\newcommand{\abs}[1]{
	\left|#1\right|}
\newcommand{\flooor}[1]{
	\Big\lfloor#1\Big\rfloor}

\DeclareMathOperator{\sgn}{sgn}

\DeclareMathOperator{\sign}{sign}

\DeclareMathOperator{\ev}{ev}

\DeclareMathOperator{\Hom}{Hom}

\DeclareMathOperator{\area}{area}

\DeclareMathOperator{\integers}{int}

\DeclareMathOperator{\lk}{lk}
\DeclareMathOperator{\inte}{int}
\DeclareMathOperator{\ab}{ab}
\DeclareMathOperator{\spec}{Spec}
\DeclareMathOperator{\Span}{Span}

\newcommand{\lmfrac}[2]{\mbox{\small$\displaystyle\frac{#1}{#2}$}}

\newcommand{\lmcup}[2]{\mbox{\small$\displaystyle\bigcup\limits_{#1}^{#2}$}}

\title{The Atiyah-Patodi-Singer rho invariant and signatures of links}
\author{Enrico Toffoli}
\address{Fakult\"at f\"ur Mathematik\\ Universit\"at Regensburg\\   Germany}
\email{enricotoffoli@gmail.com}
\begin{document}

\begin{abstract}
	Relations between the Atiyah-Patodi-Singer rho invariant and signatures of links have been known for a long time, but they were only partially investigated. 
	In order to explore them further, we develop a versatile cut-and-paste formula for the rho invariant, which allows us to manipulate manifolds in a convenient way.
	With the help of this tool, we give a description of the multivariable signature of a link $L$ as the rho invariant of some closed $3$-manifold $Y_L$ intrinsically associated to $L$. 
We study then the rho invariant of the manifolds obtained by Dehn surgery on a $L$ along integer and rational framings.
Inspired by related results of Casson and Gordon and Cimasoni and Florens, we give formulas expressing this value as a sum of the multivariable signature of $L$ and some easy-to-compute extra terms.
\end{abstract}
\maketitle
\section{Introduction}
Given a closed, oriented manifold $N$ of odd dimension, together with a representation $\alpha\colon \pi_1(N)\to U(n)$ for some $n\in \N$, we can consider the Atiyah-Patodi-Singer rho invariant $\rho_\alpha(N)$. This is a real number defined as the difference between the eta invariant of the twisted and the untwisted odd signature operator associated to any Riemannian metric on $N$, and it turns out to be independent of the choice of the metric \cite{aps1, aps2}. 
The rho invariant is thus defined as a spectral invariant, but it has has the following fundamental property that relates it to well-studied topological invariants: if $N$ is the boundary of a compact, oriented manifold $M$ such that the representation $\alpha$ extends to $\pi_1(M)$, then it can be computed as
\begin{equation}\label{eqintroaps}
\rho_\alpha(N) =n \, \sigma(M) - \sigma_\alpha(M),
\end{equation}
where $\sigma(M)$ denotes the ordinary signature of $M$, and $\sigma_\alpha(M)$ is the twisted signature associated to the chosen extension $\alpha\colon \pi_1(M)\to U(n)$. 

In knot theory, rho invariants are used to give descriptions and generalizations of knot and link signatures. One of the basic observations is that, if 
$S_K$ is the closed manifold obtained by $0$-framed surgery on $K$ and
$\alpha\colon \pi_1(S_K)\to U(1)$ is the representation sending the meridian of $K$ to $\omega\in U(1)$, then 
\begin{equation} \label{eqintroknot}
\rho_\alpha(S_K)=-\sigma_K(\omega),
\end{equation}
where $\sigma_K$ is the Levine-Tristram signature function of $K$.
 Atiyah-Patodi-Singer rho invariants associated to higher-dimensional non-abelian representations were used in knot theory by Levine
\cite{levine, levine2} and Friedl \cite{friedlknots, friedllinks} as obstructions to knot and link concordance.
Multivariable signatures of links, defined by Cimasoni and Florens \cite{cf} using generalized Seifert surfaces, were given a description as twisted signatures of $4$-manifolds and also employed as concordance invariants \cite{cf,cnt, dfl}. However, a description of them as Atiyah-Patodi-Singer rho invariants was not investigated thoroughly. One of the goals of this paper is to fill this gap.

Given the difficulty in computing rho invariants directly, it can be very useful to study their behaviour under modifications of the original manifold. One possible way to simplify manifolds is what we will call \emph{cut-and-paste} through the rest of the paper. If a closed $(2k-1)$-dimensional manifold is split by a codimension-one closed manifold $\Sigma$ into a union $X_1\cup_\Sigma X_2$, we can often find a manifold $X_0$ with $\partial X_0=-\Sigma$ such that $X_1\cup_\Sigma X_0$ and $-X_0\cup_\Sigma X_2$ are  ``simpler'' than $X_1\cup_\Sigma X_2$. Schematically, we write this modification as
\[ X_1\cup_\Sigma X_2 \quad  \leadsto \quad X_1\cup_\Sigma X_0 \:\:\sqcup \:\: -X_0\cup_\Sigma X_2.  \]
It is then useful to compare the rho invariant of $X_1\cup_\Sigma X_2$ with the sum of the rho invariants of the two other manifolds.
Recall that, when $X_i$ is any of the three manifolds $X_0$, $X_1$ or $X_2$, the subspace
	\[
	V_{X_i}^{\alpha}=\ker(H_{k-1}(\Sigma;\mathbb {C}^n_{\alpha}) \to H_{k-1}( X_i;\mathbb {C}^n_{\alpha})).
	\]
is Lagrangian in $H_{k-1}(\Sigma;\C^n_\alpha)$ with respect to the symplectic form given essentially by the twisted intersection form (we will omit $\alpha$ from the notation when we are considering the trivial $1$-dimensional representation). In particular, we can compute their \emph{Maslov triple index}
$\tau(V^\alpha_{X_0},V_{X_1}^\alpha,V^\alpha_{X_2})$, which is an integer-valued function defined on triples of Lagrangian subspaces (see Definition \ref{defmaslov}). The cut-and-paste formula for the Atiyah-Patodi-Singer rho invariant reads as follows.
\begin{reptheorem}{theocutandpaste}
	Let $X_1$, $X_2$ and $X_0$ be compact, oriented manifolds of dimension $2k-1$ with $\partial X_1=-\partial X_0=-\partial X_2$, and let $\alpha\colon \pi_1(X_1\cup_\partial X_2)\to U(n)$ be a representation that extends to $\pi_1(X_0)$. Then, for every choice of such an extension, we have
	\[
	\rho_\alpha(X_1\cup_\Sigma X_2) = \rho_\alpha(X_1 \cup_\Sigma X_0) + \rho_\alpha(-X_0\cup_\Sigma X_2)+
	C,
	\]
	where 
	\[C=\tau(V^\alpha_{X_0},V_{X_1}^\alpha,V^\alpha_{X_2}) - n\, \tau(V_{X_0},V_{X_1},V_{X_2}),\]
	with the first Maslov triple index  performed on $H_{k-1}(\partial X_1;\C^n_\alpha)$, and the second on $H_{k-1}(\partial X_1;\C)$.
\end{reptheorem}
An analogous cut-and-paste formula for the untwisted eta invariant was proved by Bunke \cite[2.5]{bunke}. Formulas related to Theorem~\ref{theocutandpaste} were also discussed by Kirk and Lesch in connection with their gluing formulas for eta invariants for manifolds with boundary \cite[Section~8.3]{kl1}. In fact, our cut-and-paste formula can be proved using their results (see \cite[Section~2.3.4]{et}). Here, however, we give a simple proof which does not involve rho invariants with manifolds with boundary, and it is based on Wall's non-additivity for the signature instead.

The second part of the paper is made up of applications of Theorem~\ref{theocutandpaste} in the context of link theory, which help relate multivariable signatures to rho invariants. 
 We consider from now on only rho invariants of $3$-manifolds with $1$-dimensional representations of the fundamental group. As such representations factor through the first homology in a unique way, we can simply see them as representations of the first homology group. 
 We recall that a $k$-component link $L=K_1\cup \cdots \cup K_k$ is said to be $n$-\emph{colored} if it is considered together with a surjective map $c\colon \{1, \dotsc , k\}\to \{1, \dotsc , n\}$, which partitions it naturally into $n$ sublinks $L_1, \dotsc , L_n$. A component $K_i$ is then said to have color $c(i)$.
 Given an $n$-colored link $L$ in $S^3$, the Cimasoni-Florens signature is a function
\[\sigma_L\colon \T_*^n \to \Z,\]
where $\T_*^n:=(S^1\setminus \{1\})^n$.
If $n=1$, $\sigma_L$ coincides with the Levine-Tristram signature function. Now, let $X_L$ be the exterior of $L$, i.e.\ the complement of an open tubular neighborhood of $L$ in $S^3$. Then, $X_L$ is a compact, oriented $3$-manifold with boundary, whose first homology group is a free abelian group generated by the meridians of $L$. Observe that $\T_*^n$ has a natural bijection with the set of representation $H_1(X_L;\Z^n)\to U(1)$ that send the meridians of components of the same color $s$ to a same value $\omega_s\in S^1\setminus \{1\}$ (we call these \emph{colored representations}): given an element $\omega=(\omega_1, \dotsc, \omega_n)\in \T_*^n$, we have an \emph{associated} colored representation $\alpha$ defined by sending the meridian of a component $K_i$ to to $\omega_{c(i)}$. The following result, which can be seen as a generalization of \eqref{eqintroknot}, expresses the multivariable signature of $L$ as the rho invariant of a suitable closed $3$-manifold $Y_L$ which only depends on $L$.  This manifold $Y_L$ is built by gluing the link exterior $X_L$ together with a $3$-manifold with boundary obtained by plumbing punctured disks in a way that is prescribed by the linking numbers of $L$. For a precise definition of $Y_L$, see Construction \ref{construction}.
       \begin{repproposition}{extheomultisign}
	          Let $L$ be an $n$-colored link. Let $\omega\in \T_*^n$, and let $\alpha\colon H_1(X_L;\Z)\to U(1)$ be the associated colored representation. Then, $\alpha$ can be extended to a representation of $H_1(Y_L;\Z)$ and, for any choice of an extension, we have 
	          \[ \rho_\alpha(Y_L) = - \sigma_L(\omega).\]
\end{repproposition}
The manifold $Y_L$ has a very simple description in the case when $L$ is \emph{color-to-color algebraically split}, i.e.\ when the total linking number $\lk(L_s,L_t)$ is $0$ for every pair of distinct colors $s,t$. Under this assumption, we prove two formulas relating the rho invariants of the closed manifolds obtained by Dehn surgery on $L$ with the multivariable signature of $L$. 
Given a link $L=K_1\cup\cdots \cup K_k$ with a rational framing $r=(r_1, \dotsc, r_k)\in \Q^k$, let $\Lambda_r$ be the rational matrix with coefficients
\[\Lambda_{ij}=
\begin{cases} 
\lk (K_i, K_j), \quad &\text{if }i\ne j,\\
r_i, \quad &\text{if } i=j.
\end{cases}\]
As $\lambda_r$ is a symmetric matrix, we can compute its signature $\sign \Lambda_r$.
Let $S_L(r)$ be the closed $3$-manifold obtained by Dehn surgery on $L$ along the rational framing $r$. We say that a representation $\alpha\colon H_1(X_L;\L)\to U(1)$ is compatible with $r$, if it extends to $H_1(S_(r);\Z)$.
First, we focus on surgery along integral framings. 
In this context, a fairly general formula was given by Cimasoni and Florens \cite[Theorem~6.7]{cf}, extending to the multivariable setting a result of Casson and Gordon \cite[Lemma~3.1]{cassongordon2} (see also \cite[Theorem~3.6]{gilmer}). 
These formulas only consider representations with a finite image, and are written in terms of the invariant $\sigma(N,\alpha)$ of Casson and Gordon (see Section~\ref{sscassgord}). 
Rewritten in terms of the rho invariant, the result of Cimasoni and Florens reads as follows.
 \begin{reptheorem}{theocf}[Cimasoni-Florens]
	Let $L$ be a $k$-colored $k$-component link. Let $q\in \N$ a positive integer and let $n_1,\dotsc, n_k\in \{1,\dotsc , p-1\}$ be integers, each of which is coprime with $q$.  Let $\omega=(e^{2\pi in _1/q}, \dotsc , e^{2\pi i n_k/q})\in \T_*^n$, and let $\alpha \colon H_1(X_L;\Z)\to U(1)$ be the associated colored representation. Let $g$ be a compatible integral framing on $L$. Then, we have
	\[
	\rho_\alpha(S_L(g))= -\sigma_L(\omega) +\sum_{i<j}\Lambda_{ij} +  \sign \Lambda_g  -\lmfrac{2}{q^2}\sum_{i=1}^k (q-n_i)n_j \Lambda_{ij}.
	\]
\end{reptheorem}
Theorem~\ref{theocf} is stated for the signature associated to the \emph{maximal coloring}, where every distinct link components have different colors, but a formula for any coloring can be easily deduced from it. However, there are some restrictions on the values of $\omega$ that limit its use. Under the assumption of $L$ being color-to-color algebraically split, we prove the following formula with no restrictions on the values of $\omega$.
 \begin{reptheorem}{theointeger}
	Let $L$ be an $n$-colored link which is color-to-color algebraically split. Let $\omega\in \T_*^n$, and let $\alpha \colon H_1(X_L;\Z)\to U(1)$ be the associated colored representation. Let $g$ be a compatible integral framing on $L$.
	Then, we have
	\[
	\rho_\alpha(S_L(g))= -\sigma_L(\omega) +  \sign \Lambda_g -  2\sum_{s=1}^n h_s\theta_s(1-\theta_s),\]
	where, for each color $s$, the values $h_s$ and $\theta_s$ are determined by
		\[h_s:=\sum_{c(i)=c(j)=s} \Lambda_{ij}, \quad \quad \theta_s\in(0,1) \text{ such that } \omega_s=e^{2\pi i s}.\]
\end{reptheorem}
Theorem~\ref{theointeger} is proved by first describing the multivariable signature as a rho invariant using Proposition~\ref{extheomultisign}, and then applying the cut-and-paste formula (Theorem~\ref{theocutandpaste}) to modify $Y_L$ into a disjoint union of $S_L(g)$ and lens spaces $L(h_s,1)$, for which we can write the rho invariant very explicitly. The proof is completed by a careful computation of the Maslov triple index involved in the cut-and-paste formula.

Suppose now that a stronger assumption on the linking numbers is satisfied, namely that every link component $K_i$ has total linking number $0$ with all sublinks $L_s$ such that $s\ne c(i)$. We say in this case that $L$ is \emph{component-to-color algebraically split}. Then, as it is easy to verify, there is a particular integer framing $f_L$ (the \emph{Seifert framing}) which is compatible with all colored representations $H_1(X_L;\Z)\to U(1)$. An immediate consequence of Theorem~\ref{theointeger} is the following.
\begin{repcorollary}{corseifert}
	Let $L$ be an $n$-colored link which is component-to-color algebraically split. Let $\omega\in \T_*^n$, and let $\alpha \colon H_1(X_L;\Z)\to U(1)$ be the associated colored representation. Then, $\alpha$ extends to $H_1(S_L(f_L);\Z)$ and we have
	\[
	\rho_\alpha(S_L(f_L))= -\sigma_L(\omega) +  \sign \Lambda_{f_L}.
	\]
\end{repcorollary}
The analogue to Corollay \ref{corseifert} for the $1$-colored setting was used by Nagel and Powell in studying concordance properties of the Levine-Tristram signature \cite{nagelpowell}. Besides being a useful formula on its own, moreover, it is our starting point to prove a result that takes into account (non-integral) rational framings. This is expressed by the next and final result of this paper.

 \begin{reptheorem}{theorational}
	Let $L$ be an $n$-colored, $k$-component link that is component-to-color algebraically split. Let $\omega\in \T_*^n$, and let $\alpha\colon\pi_1(X_L)\to U(1)$ be the associated colored representation. Let $r$ be a compatible rational framing on $L$. Then, we have
	\[
	\rho_\alpha(S_L(r))= -\sigma_L(\omega) +  \sign \Lambda_r - \sum_{i=1}^k (\rho(L(p_i,q_i),\omega_{c(i)}) + \sgn(p_i/q_i)) ,
	\]
	where $p_i,q_i$ are coprime integers such that $r_i-f_i=p_i/q_i$ (here $f_i$ is the $i$-th coefficient of the Seifert framing).
\end{reptheorem}
Observe that any $1$-colored link is trivially component-to-color algebraically split. In particular, Theorem~\ref{theorational} gives a new formula relating rho invariants with the Levine-Tristram signature.
As for Theorem~\ref{theointeger}, lens spaces arise from the cut-and-paste construction. In this case, we do not spell out the values of their rho invariant in the statement of the theorem, as this would make the formula more cumbersome. Note however that these value can be always computed easily (see Proposition~\ref{proplens}).

\begin{remark*}
Whenever not stated otherwise, all manifolds are assumed to be smooth.
\end{remark*}

\subsubsection*{Outline of the paper}
In Section~2, we review the basics about twisted signatures and Atiyah-Patodi-Singer rho invariants. We illustrate then a formula for the rho invariant of a $3$-dimensional lens space.
In Section~3, we review the Maslov triple index of Lagrangian subspaces and Wall's non-additivity theorem for the signature. We prove then our cut-and-paste formula (Theorem~\ref{theocutandpaste}).
In section 4, we develop the applications in knot theory. We first prove the basic formula relating rho invariants and multivariable signatures (Proposition~\ref{extheomultisign}), and then use it to show results about integer (Theorem~\ref{theointeger}) and rational (Theorem~\ref{theorational}) Dehn surgery.

\subsubsection*{Acknowledgements}
This project was supported by the collaborative research center SFB 1085 ``Higher Invariants'', funded by the Deutsche Forschungsgemeinschaft. Part of the article is based on the author's PhD thesis, which was written under the support of the graduate school GRK 1692 ``Curvature, Cycles, and
Cohomology'', also funded by the Deutsche Forschungsgemeinschaft.
The author would like to thank Stefan Friedl for several interesting discussions.

\section{Twisted signatures and rho invariants}
In Section~\ref{sstwisted} we review the definition of twisted homology and twisted signatures, and set notation for these. 
In Section~\ref{ssrho}, we recall the basics about Atiyah-Patodi-Singer eta and rho invariants.
In Section~\ref{sscassgord}, we underline the relation between these invariants and an invariant of Casson and Gordon.
In Section~\ref{sslens}, we give a reinterpretation of a well-known computation for the rho invariant of $3$-dimensional lens spaces. 

\subsection{Twisted intersection forms and signatures}\label{sstwisted}
Let $M$ be a connected, compact, oriented manifold of dimension $2k$ with a representation $\alpha\colon \pi_1(M) \to U(n)$ for some $n\in \N$. 
Let $\pi:=\pi_1(M)$. We consider the twisted homology groups
\[H_i(X;\C^n_{\xi^\alpha}):= H_i(C_*(\widetilde{X})\otimes_{\Z[\pi]}\C^n ),\]
where $\widetilde{X}$ is the universal cover of $X$, so that $\pi$ acts on $C_*(\widetilde{X})$ from the right by (reversed) deck transformations (i.e.\ $(\sigma \cdot g)(x):=g^{-1}\cdot \sigma(x)$), and it acts on $\C^n$ from the left through the representation $\alpha$.
\begin{remark}
If $M=M_1\sqcup \cdots \sqcup M_N$ is a disjoint union of connected manifolds $M_j$, we will make the abuse of notation of writing $\alpha\colon \pi_1(M) \to U(n)$ to denote a collection of representations $\alpha_j\colon \pi_1(M_j)\to U(n)$. In this case, we define then
\[H_i(X;\C^n_{\xi^\alpha}):= \bigoplus_{j=1}^N H_i(X_j;\C^n_{\xi^\alpha_j}).\]
\end{remark}

In a similar way, twisted cohomology groups can be defined, and a twisted Poincar\'e duality isomorphism is satisfied.
The ordinary intersection pairing on $M$ can be generalized into a sesquilinear form
\[
I^\alpha_M \colon  H_k(M;\C^n_\alpha)\times H_{k}(M;\C^n_\alpha)\to \C, 
\]
called the \emph{twisted intersection form},
which arises from the sequence of maps
\[
H_k(M;\C^n_\alpha)\xrightarrow{j_*}H_k(M,\partial M;\C^n_\alpha) \xrightarrow{PD} H^k(M;\C^n_\alpha)\xrightarrow{\ev}\overline{\Hom (H_k(M;\C^n_\alpha),\C)},
\]
where the bar denotes complex conjugation of vector spaces
(compare e.g.\ with \cite[Definition 2.7]{cnt}).
The form $I_M^\alpha$ is not necessarily non-degenerate, but it is non-degenerate if $M$ is closed. In fact, its radical is exactly the kernel of $j_*$, as the two other maps in the composition are isomorphisms. 
 Moreover, $I_M^\alpha$ is Hermitian if $k$ is even and skew-Hermitian if $k$ is odd.
 This leads to the following definition.
 \begin{definition}
 	Let $M$ be a compact oriented manifold of dimension $2k$ with a representation $\alpha\colon \pi_1(M)\to U(n)$.
The \emph{signature of $M$ twisted by $\alpha$} is the integer
 \[\sigma_\alpha(M):=\begin{cases}\sign (I_M^\alpha) , \text{ if $k$ is even},\\
 \sign(i\,I_M^\alpha),  \text{ if $k$ is odd},
 \end{cases}\]
 where $I_M^\alpha$ denotes the intersection pairing in degree $k$.
 \end{definition}
 
%
%
%

\subsection{Basics on the rho invariant}\label{ssrho}

Let $N$ be a closed, oriented, Riemannian manifold of dimension $2k-1$, with a representation $\alpha\colon\pi_1(N)\to U(n)$. 
Let $E_\alpha\to N$ be the associated flat vector bundle, and consider the subspace
\[\Omega^{\ev}(N, E_\alpha):=\bigoplus_{q=0}^{k-1} \Omega^{2q}(N, E_\alpha)\]
of twisted differential forms of even degree. Let $D^\alpha_N$ be the \emph{twisted odd signature operator}, i.e.\ the first-order differential operator on $\Omega^{\ev}(N, E_\alpha)$ defined by
\[ D^\alpha_N \,\phi := (-1)^{q+1}i^k(\star d - d \star)\phi, \quad \text{ for  } \phi\in \Omega^{2q}(N, E_\alpha).\]
The operator $D^\alpha_N$ can be extended to a self-adjoint elliptic operator with discrete spectrum and, by the results of Atiyah, Patodi and Singer \cite{aps1}, it has a well defined \emph{eta invariant}
\[\eta_\alpha(N):=\eta( D^\alpha_N)\in \R,\]
which is defined as the value at $0$ of a meromorphic extension of the eta function
\[ \eta(s)=\sum_{\substack{ \lambda\in \spec (D^\alpha_N)\\ \lambda\ne 0} } \!\!\!\!\!\sgn \lambda\, \lvert \lambda \rvert ^{-s}.\]
We say that a compact Riemannian manifold $M$ has metric of \emph{product form} near the boundary, if there exists a neighborhood of $\partial M$ that is isometric to $(-\varepsilon, 0] \times \partial M$ with the product metric.
The main result about the eta invariant of the twisted signature operator is the following \cite[Theorem~2.2]{aps2}.
\begin{theorem}[Atiyah-Patodi-Singer] \label{theoaps} Let $M$ be a compact, oriented manifold with $\partial M=N$, equipped with Riemannian metric of product form near $N$, and let $\alpha\colon\pi_1(M)\to U(n)$ be a representation. Then
	\[	\sigma_\alpha(M)= n\int_M \!L(p)  -\eta_\alpha(N) ,  \]
	where $L(p)$ is the Hirzebruch $L$-polynomial in the Pontryagin forms of $M$.
\end{theorem}
Note that both summands on the right-hand term depend on the Riemannian metric on $N$.
We shall not dwell upon the geometrical significance of the integral of the $L$-polynomial, as it is going to get simplified soon.
\begin{remark}
	The restriction of $\alpha\colon \pi_1(M) \to U(n)$ to a representation of $\pi_1(N)$ is made by composing $\alpha$ with the natural map $\pi_1(N)\to \pi_1(M)$. 
	If $N$ is not connected, a map $\pi_1(N_1)\to \pi_1(M)$ for each connected component $N_i$ of $N$ can be obtained by choosing appropriate paths between base points. 
The restriction $\alpha\colon \pi_1(N)\to U(n)$ must be interpreted as the collection of the representations $\pi_1(N_i)\to U(n)$ for all connected components $N_i$. The invariant $\eta_\alpha(N)$ is defined in this case as the sum of the eta invariants of the $N_i$'s.
\end{remark}
We are now going to define the rho invariant. Note that, for the untwisted odd signature operator $D_N$ on $\Omega^{\ev}(N, \C)$, we set $\eta(N):=\eta(D_N)$.
\begin{definition}
	Let $N$ be a closed, oriented manifold of odd dimension, and let $\alpha\colon\pi_1(N)\to U(n)$ be a representation. The \emph{Atiyah-Patodi-Singer rho invariant} of $N$ associated to $\alpha$ is the real number
	\[\rho_\alpha(N):=\eta_\alpha(N)-n\,\eta(N),\]
	where the eta invariants are computed for an arbitrary Riemannian metric on $N$.
\end{definition}

We shall see in a moment that the difference $\eta_\alpha(N)-n\,\eta(N)$ is independent of the Riemannian metric, so that $\rho_\alpha(N)$ is well defined.   Moreover, $\rho_\tau(N)$ is $0$ for trivial representations $\tau$, and it satisfies
\begin{equation}\label{eqapsrhoorientation}
\rho_\alpha(-N)=-\rho_\alpha(N).
\end{equation}

The main theorem about the rho invariant is the following \cite[Theorem~2.4]{aps2}, which is an easy consequence of Theorem~\ref{theoaps}.
\begin{theorem}[Atiyah-Patodi-Singer]~ \label{theoapsrho}
\begin{enumerate}[(i)]
		\item $\rho_\alpha(N)$ is independent of the Riemannian metric on $N$. 
	\item If $M$ is a compact, oriented manifold with $\partial M=N$ and $\alpha$ extends to $M$, then
	\[\rho_\alpha(N)= n\,\sigma(M)-\sigma_\alpha(M). \]
\end{enumerate}
\end{theorem}

We state one more result that will turn useful later on. 
\begin{proposition}\label{proprhoproduct}
	Let $\Sigma$ be a closed, oriented surface, and let $\psi \colon\pi_1(\Sigma\times S^1)\to U(1)$ be a representation. Then
$\rho_\psi(\Sigma\times S^1)=0.$
\end{proposition}
\begin{proof}
See \cite[Lemma~4.2]{cnt}.
\end{proof}

\subsection{Rho invariants and Casson-Gordon invariants}\label{sscassgord}
We will now review the definition of an invariant of Casson and Gordon, and relate it to the Atiyah-Patodi-Singer rho invariant.
Let $N$ be a closed, oriented $3$-manifold, and let $\alpha\colon H_1(N;\Z)\to U(1)$ be a representation. Assume that the image of $\alpha$ is finite. 
Using a bordism argument, Casson and Gordon observe that there exists a compact, oriented $4$-manifold $W$ such that the boundary of $W$ is the disjoint union of $r$ copies of $N$ for some $r\in \N$ (we will write $\partial W =r N$) with a representation $\alpha'\colon H_1(W;\Z)\to U(1)$ that restrict to $\alpha$ on each boundary component. They define then an invariant as 
\begin{equation}\label{eqsigma}
\sigma(N,\alpha):=\lmfrac{1}{r}(\sigma_{\alpha'}(W)-\sigma(W)).
\end{equation}
By additivity of the signature and again some bordism theory, they show that the invariant $\sigma(N,\alpha)$ is independent of the choice of $W$ and of the extension $\alpha'$ (see also \cite[Corollary~2.11]{cnt} for a more detailed version of their proof).
Using the Atiyah-Patodi-Singer index theorem, their invariant can be immediately reinterpreted as an Atiyah-Patodi-Singer rho invariant. We state this explicitly for further reference, albeit it is surely known to the experts.
\begin{proposition}\label{propcassgord}
Let $N$ be a closed, oriented $3$-manifold, and let $\alpha\colon H_1(N;\Z)\to U(1)$ be a representation with finite image. Then, we have
\[\sigma(N,\alpha)=- \rho_\alpha(N).\]
\end{proposition}
\begin{proof}
Let $W$ be a compact, oriented $4$-manifold with $\partial W = rN$, with a representation $\alpha'\colon H_1(W;\Z)\to U(1)$ that restricts to $\alpha$ as discussed above, so that $\sigma(N,\alpha)$ is described by \eqref{eqsigma}. 
Using Theorem~\ref{theoapsrho}, on the other hand, we have
\[r\rho_{\alpha'}(N)=\rho_{\alpha'}(rN)=\sigma(W)-\sigma_\alpha(W).\] 
Comparing this with \eqref{eqsigma}, we obtain the desired statement.
\end{proof}

\subsection{The rho invariant of lens spaces}\label{sslens}
Given coprime integers $p$ and $q$, the $3$-dimensional lens space $L(p,q)$ can be built as the union of two solid tori $Y_1,Y_2$ along any orientation-reversing diffeomorphism $f\colon \partial V_2\to \partial V_1$ such that 
\[f_*(\mu_2)=  - q \mu_1 + p\lambda_1,\]
where $m_1, m_2$ are the meridians respectively of $Y_1, Y_2$, and $l_1$ is a longitude of $Y_1$.
This construction is well defined for both positive and negative values of $p$ and $q$. For positive $p$ it coincides, up to some explicit orientation-preserving diffeomorphism, with the classical definition of $L(p,q)$ as a quotient of $S^3$ (compare with \cite[Lemma~91.3]{friedlat}). 
In general, there are orientation preserving diffeomorphisms
\[L(-p,q)\cong L(p,-q) \cong - L(p,q).\]
In particular, for both positive and negative $p$ we have an identification of
$\pi_1(L(p,q))$ with $\Z/p$, and under this identification the element $[1]\in \Z/p$ coincides with the image of the generator of $\pi_1(Y_1)$ under the gluing.

Rho invariants of $3$-dimensional lens spaces can be computed explicitly.
As every representation $\alpha\colon \Z/p\to U(n)$ can be written as a direct sum of $1$-dimensional representations, we shall focus on $1$-dimensional representations.
Moreover, we shall exclude the case $p=0$, as $L(0,1)$ is diffeomorphic to $S^2\times S^1$ and its rho invariant is $0$ for any choice of $\alpha$.
We observe in this case that the representations $\alpha\colon \Z/p\to U(1)$
are in a natural bijection with the set of $\abs{p}$\textsuperscript{th} roots of unity: to each such root $\omega$, we associate the representation $\alpha_\omega$ sending $1$ to $\omega$. 
\begin{notation}Given a $\abs{p}$\textsuperscript{th} root of unity $\omega$, we write 
	\[\rho(L(p,q),\omega):= \rho_{\alpha_\omega}(L(p,q)).\]
\end{notation}

Formulas for the rho invariants of lens spaces were given since the original paper of Atiyah, Patodi and Singer \cite[Proposition~2.12]{aps2}. Introducing the periodic sawtooth function $\saw{\cdot}\colon \R\to (-\frac{1}{2}, \frac{1}{2})$ defined by
\[ \saw{x}:= \begin{cases} x- \lfloor x\rfloor - \lmfrac{1}{2}, &\text{ if }x\in \R\setminus \Z,\\
0,    &\text{ if }x\in \Z.\end{cases},\] we give the following description.
\begin{proposition}\label{proplens}
Let $p,q$ be two coprime integers with $p\ne 0$, and let $\zeta=e^{2\pi i /p}$. Then, for $k\in \{0,1,\dotsc, \abs{p}-1\}$, we have
	\[\rho(L(p,q),\zeta^{kq})=-4\sum_{j=1}^{k-1} \saww{\lmfrac{qj}{p}} - 2\saww{\lmfrac{qk}{p}}.
\]
\end{proposition}
\begin{proof}
We first suppose $p,q>0$. Set $z:=e^{2\pi i n/p}$, with $n=\gcd(p,k)$, and set $r:=k/n$. Then, we have $\zeta^{kq}=z^{rq}$.
The representation sending $1$ to $\omega^q$ has as its image the set of $m$\textsuperscript{th} roots of unity, with $m$ such that $p=mn$. In this setting, Casson and Gordon \cite[pp.187-188]{cassongordon} gave a formula for the invariant $\sigma(L(p,q),\alpha_\omega)$ that using Proposition~\ref{propcassgord} we can rewrite as
\begin{equation}\label{eqprooflattice}
\rho(L(p,q),  z^{rq})= -4\Big(\area\Delta\Big(nr,\lmfrac{rq}{m}\Big)- \integers\Delta\Big(nr,\lmfrac{rq}{m}\Big)\Big)
,
\end{equation}
where
$\Delta(x,y)$ is the triangle with vertices $(0,0)$, $(x,0)$ and $(x,y)$ and the number $\integers(\Delta(x,y))$ is given by counting:
\begin{itemize}
\item $+1$ for every point of $\Z^2$ that lies in the interior of $\Delta(x,y)$;
\item $+1/2$ for every point of $\Z^2$ that lies in the interior of its edges;
\item $+1/4$ for every point of $\Z^2\setminus \{(0,0)\}$ that coincides with one of the vertices.
\end{itemize}
We observe that \eqref{eqprooflattice} can be immediately rewritten as 
\begin{equation}\label{eqprooflattice2}
\rho(L(p,q),\zeta^{kq})=4\Big(\integers\Delta\Big(k,\lmfrac{kq}{p}\Big)-\area\Delta\Big(k,\lmfrac{kq}{p}\Big)\Big).
\end{equation}
We will now express the right-hand term of \eqref{eqprooflattice2} in a more explicit way. First of all, it is clear that $4\area\Delta\Big(k,\lmfrac{kq}{p}\Big)= \lmfrac{2q}{p}k^2. $
Moreover, we can count the lattice points inside the triangle by following vertical lines $\{(x,y)\,\vert\,x=j\}$, for $j=1,\dotsc , k$, and then summing over $j$. We obtain
\[\begin{split}
4\integers\Delta\Big(k,\lmfrac{kq}{p}\Big)&= 4\sum_{j=1}^{k-1}\Big(\lmfrac{1}{2}+\flooor{\lmfrac{jq}{p}}\Big) +4\Big(\lmfrac{1}{4}+\lmfrac{1}{2}\flooor {\lmfrac{kq}{p}}\Big)=\\
&=2k -1 +4\sum_{j=1}^{k-1} \flooor{\lmfrac{jq}{p}} +2\flooor{\lmfrac{kq}{p}}.
\end{split} \]
Taking the difference, we obtain thus
\begin{equation}\label{eqlensspacesold}
\rho(L(p,q),\zeta^{kq})=-\lmfrac{2q}{p}k^2 +2k -1 +4\sum_{j=1}^{k-1}\flooor{ \lmfrac{jq}{p}} +2\flooor{\lmfrac{kq}{p}}. 
\end{equation}
Expanding the expression in the statement, it is now immediate to see that it coincides with \eqref{eqlensspacesold}.
As the sawtooth function $\saw{\cdot}$ is odd, we see that both sides of the identity change sign when either $p$ or $q$ is changed of sign. As a consequence, the result keeps holding for non-positive choices of $p$ and $q$.
\end{proof}

\begin{corollary} \label{corlensinteger}
Let $n$ be any integer, and let $\omega\in U(1)$ be an $n$\textsuperscript{th} root of unity. Then, we have 
	\[
\rho(L(n,1),\omega)= 2n\theta(1-\theta) -\sgn (n),
\]	
where $\theta\in [0,1)$ is such that $\omega=e^{2\pi i\theta}$.
\end{corollary}
\begin{proof}
	For $n=0$, we have $\rho(L(n,1),\omega)=0$ for all $\omega\in U(1)$ because $L(0,1)=S^2\times S^1$, so that the result is trivially satisfied in this case.
	For $n>0$, there has to be a $k\in\{0,1, \cdots n-1\}$ such that $\theta=k/n$, with $k\in\{0,1, \cdots n-1\}$. From  \eqref{eqlensspacesold} we easisy see that
		\[
	\rho(L(n,1),\omega)=-\lmfrac{2k^2}{n} +2k -1 = 2n\theta(1-\theta) -1,
	\]
	and the desired formula is satisfied in this case.
For negative $n$ we obtain now from the last equation that
	\[
\rho(L(n,1),\omega)=-\rho(L(-n,1),\omega)=2n\theta(1-\theta) +1,
\]
which leads to the general formula of the statement.
\end{proof}

\begin{remark}Proposition~\ref{proplens} can be used to write rho invariants of lens spaces as a difference of Dedekind-Rademacher sums. See \cite[Theorem~3.3.27]{et}.
\end{remark}

\section{The cut-and-paste formula}
In Section~\ref{ssmaslov} we recall the definition of the Maslov triple index of three Lagrangian subspaces in a complex symplectic space.
In Section~\ref{sswall}, we review a (non-)additivity theorem of Wall for signatures of manifolds under some fairly general notion of gluing.
In Section~\ref{sscutandpasterho}, we prove the cut-and-paste formula for the Atiyah-Patodi-Singer rho invariants (Theorem~\ref{theocutandpaste}), which is the main result of this section.

\subsection{Complex symplectic spaces and the Maslov triple index}\label{ssmaslov}
A \emph{complex symplectic space} is a pair $(H,\omega)$ such that $H$ is a finite-dimensional complex vector space, and $\omega\colon H\times H\to \C$ is a non-degenerate skew-Hermitian form, called the \emph{symplectic form}. We shall often omit $\omega$ from the notation and simply call $H$ a complex symplectic space.
We recall that a subspace $L\subseteq H$ is \emph{Lagrangian} if it coincides with its orthogonal complement with respect to the symplectic form $\omega$. 
Let $\LL ag(H)$ denote the set of its Lagrangian subspaces. 

In the applications, complex symplectic spaces will arise from the following setting.
Let $\Sigma$ be a $(2k-2)$-dimensional closed, oriented manifold, and let $\alpha\colon \pi_1(\Sigma)\to U(n)$ be a representation. As we have seen, the twisted intersection form on $H:=H_{k-1}(\Sigma;\C^n_\alpha)$ is skew-Hermitian if $k$ is even, and it is Hermitian if $k$ is odd. Moreover, it is non-degenerate because $\Sigma$ is closed. We can thus always consider the non-degenerate, skew-Hermitian form
\[\omega:=\begin{cases}
I_\Sigma^\alpha,  \text{ if $k$ is even},\\
i\, I_\Sigma^\alpha , \text{ if $k$ is odd},
\end{cases}\]
which makes the pair $(H,\omega)$ a complex symplectic space.
We introduce the following notation.
\begin{notation}\label{notacanlag}
	Given a compact, connected $(2k-1)$-dimensional manifold $X$, we set
	\[
	V_X^{\alpha}=\ker(H_k(\partial X;\mathbb {C}^n_{\alpha}) \to H_k(\partial X;\mathbb {C}^n_{\alpha})).
	\]
\end{notation}
A well-known argument based on Poincar\'e duality shows that the subspace $V_X^{\alpha}$ is Lagrangian in $H_k(\partial X;\mathbb {C}^n_{\alpha})$ with respect to the symplectic form $\omega$.
\begin{definition}
	We refer to $V_X^\alpha$ as the \emph{canonical Lagrangian} associated to $X$ and $\alpha$.
\end{definition}

	Going back to the case of an abstract complex symplectic space $(H,\omega)$, we are now going to define a function
		\[\tau\colon \LL ag (H) \times \LL ag (H) \times \LL ag (H) \to \Z.\]
Given three Lagrangian subspaces $L_1, L_2, L_3\in \LL ag(H)$, it is easy to observe that the sesquilinear form 
\begin{equation*}
				\begin{split}
				\psi_{L_1L_2L_3}\colon(L_1+L_2)\cap  L_3 \times    (L_1+L_2)\cap L_3   &\to \mathbb{C} \\
				(a_1+a_2,b_1+b_2)  &\mapsto \omega(a_1,b_2)
				\end{split}
				\end{equation*}
(with $a_1,b_1\in L_1$ and $a_2,b_2\in L_2$) is well defined and Hermitian. In particular, we can give the following definition.				
	\begin{definition} \label{defmaslov}
	The \emph{Maslov triple index} of $(L_1,L_2,L_3)$ is the integer 
	\[	\tau(U,V,W):=\sign \psi_{L_1L_2L_3}.\]
	\end{definition}

	The Maslov triple index satisfies several elementary properties, among which are the following two.
\begin{proposition}\label{propmaslov}
		\begin{enumerate}[(i)]
			\item Let $L_1,L_2,L_3\in \LL ag (H)$, and let $\alpha$ be a permutation of the set $\{1,2,3\}$. Then 
\[
			\tau(L_{\alpha(1)}, L_{\alpha(2)}, L_{\alpha(3)})= \sgn(\alpha)\,\tau(L_1, L_2, L_3).
\]
			In particular, $\tau(L_1,L_2,L_3)=0$ if two of the Lagrangians coincide.
			
			\item Let $L_1,L_2,L_3, L_4\in \LL ag (H)$. Then, $\tau$ satisfies the \emph{cocycle equation}
\[
			\tau(L_1, L_2, L_3) -\tau(L_1, L_2, L_4)+ \tau(L_1, L_3, L_4)-\tau(L_2, L_3, L_4)   =0.
			\]
		\end{enumerate}
\end{proposition}
\begin{proof}See e.g.\ \cite[Section 8]{maslov} for proofs in the real symplectic setting that can be adapted verbatim to the complex symplectic one.
\end{proof}
\begin{example}\label{extaudim2}
		Suppose that $(H,\omega)$ is a complex symplectic space of dimension $2$.  Let $(\mu,\lambda)$ be an ordered basis in which $\omega$ is represented by the matrix $\left (\begin{smallmatrix} 0&-1\\1& 0\end{smallmatrix}\right )$, i.e. such that \[\omega(\mu,\mu)=\omega(\lambda,\lambda)=0, \quad \omega(\mu,\lambda)=-1\] 
		we shall call such pair a \emph{symplectic basis}. Then, it is easy to verify that a $1$-dimensional subspace is Lagrangian if and only if it is the span of some vector $a\mu+b\lambda$ with $a,b\in\R$. We set in this case the notation
		\[\tau(v_1,v_2,v_3):= \tau (\Span_\C\{v_1\}, \Span_\C\{v_2\}, \Span_\C\{v_3\}).\]
		Using the definition of the Maslov triple index, we easily compute that
	\[\tau (\mu,\lambda, a\mu+b\lambda)= - \sign(ab).\]
\end{example}

\subsection{Wall's non-additivity of the signature}\label{sswall}

We shall now review a result of Wall. We start with some notation that will be useful throughout the paper.
\begin{notation}
Given two topological spaces $X,Y$ with a common subspace $A$, we set
\[ X\cup_A Y:= (X\sqcup Y) \,/ \! \sim,\]
where $\sim$ is the relation that identifies every element in $A\subseteq X$ with its copy in $A\subseteq Y$. We say that  $X\cup_A Y$ is obtained by \emph{gluing} $X$ and $Y$ along $A$. If $X$ and $Y$ are manifolds with $\partial X= \partial Y$, we also write $X\cup_\partial Y$ to denote the gluing along their common boundary.
\end{notation}
Suppose that $M$ is a compact, oriented manifold of dimension $2k$ which is split as 
\[M=M_1\cup_{X_0} M_2\]
 along a properly embedded submanifold  $X_0$ of codimension $1$, which is allowed to have boundary $\Sigma$. Let $X_1:=\partial M_1 \setminus \inte(X_0)$ and $X_2:=\partial M_2 \setminus \inte(X_0)$. As unoriented manifolds, we have then $\partial X_1=\partial X_1 =\partial X_2 = \Sigma$ and
\begin{equation}\label{eqorientationswall}
\partial M_1= X_1\cup_\Sigma X_0,  \quad \partial M_2= X_0\cup_\Sigma X_2, \quad \partial M= X_1\cup_\Sigma X_2 .
\end{equation}
	\begin{center}
 	 \begin{tikzpicture}[scale=1]
 	 	        
 	 	         \fill [lightgray] 
 	 	         	  	 (-9,0) arc (90:-90: 2cm)
 	 	         	  	   (-9,0)  arc (90:270:3 cm and 2cm);
 	 	         	 
 	 	 \draw[thick] (-9,0) arc (90:-90: 2cm);
 	 	 \draw[thick] (-9,0)  arc (90:270:3 cm and 2cm);
 	 	 
 	 	        	 	 \node at (-9,0) {$\bullet$};
 	 	        	 	 	 \node at (-9,0.3) {$\scriptstyle{\Sigma}$};
 	 	        	 	 \node at (-9,-4) {$\bullet$};
 	 	   \draw[thick] (-9,0) -- (-9,-4);
  \node at (-9.3,-2) {$\scriptstyle{X_0}$};
  \node at (-13.5,-2) {$=$};
  \node at (-15,-2) {$M$};
 	 	        	 	  \node at (-12.3,-2) {$\scriptstyle{X_1}$};
 	 	        	 	   	  \node at (-10.5,-2) {$M_1$};
 	 	        	 	   	  \node at (-8,-2) {$M_2$};
 	 	        	 	    \node at (-6.7,-2) {$\scriptstyle{X_2}$};
 	 	        	 	    
 	 \end{tikzpicture}
 	 	\end{center}
 	 	We pick on $X_1$ the orientation coming from being a codimension $0$ submanifold of $\partial M_1$, and we give $\Sigma$ the orientation coming from being the boundary of $X_1$. 
Suppose now that $\alpha\colon\pi_1(M)\to U(n)$ is a representation.
In our setting, $\Sigma$ is (up to orientation) the common boundary of $X_0$, $X_1$ and $X_2$, so that the canonical Lagrangians $V_{X_0}^\alpha$, $V_{X_1}^\alpha$ and $V_{X_0}^\alpha$  all live in the same space $H_{k-1}(\Sigma;\C^n_\alpha)$.
In particular, it makes sense to compute their Maslov triple index. In fact, the following result holds.
\begin{theorem}[Wall's non-additivity]\label{theowall}
Let $M$ be a closed, oriented, even-dimensional manifold, and let $\alpha\colon \pi_1(M)\to U(n)$ be a representation. Then, if $M$ decomposes as $M=M_1\cup_{X_0} M_2$ as above, we have
\[
	\sigma_\alpha( M) = \sigma_\alpha (M_1) + \sigma_\alpha (M_2) - \tau(V_{X_0}^\alpha,V^\alpha_{X_1},V^\alpha_{X_2}).
\]
\end{theorem}
\begin{remark}
Theorem~\ref{theowall} was originally proved by Wall for the untwisted signature \cite{wall} (see also the paper of Py \cite[(3.2)]{py} for a more detailed proof), and it can be easily checked that the result extends to twisted signatures. See \cite{cnt, dfl, dfl2} for further references and uses of the twisted version of the theorem.
\end{remark}

\subsection{The cut-and-paste formula for the rho invariant}\label{sscutandpasterho}
Suppose to have a closed, oriented $(2k-1)$-dimensional manifold that is split by a codimension-one closed manifold $\Sigma$, yielding a decomposition $X_1\cup_\Sigma X_2$. Let $X_0$ be a compact, oriented manifold with $\partial X_0=-\Sigma$. Then, we can replace $X_1\cup_\Sigma X_2$ with the disjoint union of  $X_1\cup_\Sigma X_0$ and $-X_0\cup_\Sigma X_2$. We will call this manipulation \emph{cut-and-paste}. Schematically, we have
\[ X_1\cup_\Sigma X_2 \quad  \leadsto \quad X_1\cup_\Sigma X_0 \:\:\sqcup \:\: -X_0\cup_\Sigma X_2,  \]
and pictorially we can represent the operation as in the next figure.
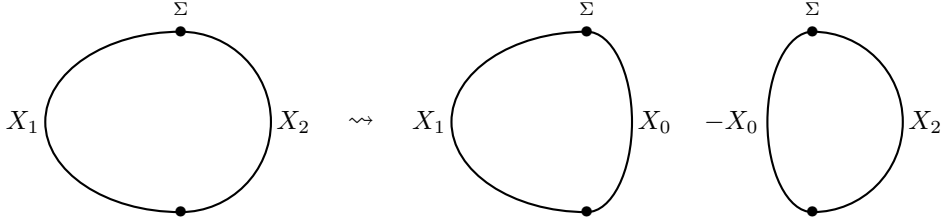
\begin{center}
	 \begin{tikzpicture}[scale=0.6]
	 	\begin{scope}[shift={(9,0)}]
 \node at (0,1) {};
	 \draw[thick] (0,0) arc (90:270: 3 cm and 2cm);
	 \draw[thick] (0,0) arc (90:-90: 1cm and 2cm);
	 \draw[thick] (5,0) arc (90:-90: 2cm);
	 \draw[thick] (5,0)  arc (90:270:1 cm and 2cm);
	 \node at (0,0) {$\bullet$};
	 \node at (0,0.5) {$\scriptstyle{\Sigma}$};
	 \node at (0,-4) {$\bullet$};
	  \node at (5,0) {$\bullet$};
	  \node at (5,0.5) {$\scriptstyle{\Sigma}$};
	 \node at (5,-4) {$\bullet$};
	 	    \node at (-3.5,-2) {$X_1$};
	 	    \node at (1.5,-2) {$X_0$};
	 	      \node at (3.2,-2) {$-X_0$};
	 	       \node at (7.5,-2) {$X_2$}; 
	 	       \end{scope}
	 	      
	 	       \node at (4,-2) {$\leadsto$};
	 	        
	 \begin{scope}[shift={(-14,0)}]        
	 	 \draw[thick] (14,0) arc (90:-90: 2cm);
	 	 \draw[thick] (14,0)  arc (90:270:3 cm and 2cm);
	 	        	 	 \node at (14,0) {$\bullet$};
	 	        	 	 \node at (14,0.5) {$\scriptstyle{\Sigma}$};
	 	        	 	 \node at (14,-4) {$\bullet$};
	 	        	 	  \node at (10.5,-2) {$X_1$};
	 	        	 	    \node at (16.5,-2) {$X_2$};
	 	   \end{scope}
	 	        	 	    
	 \end{tikzpicture}
	 	\end{center}

Suppose now that $\alpha\colon\pi_1(X_1\cup X_2)\to U(n)$ is a representation.
In particular, $\alpha$ is defined on $\pi_1(X_1)$ and on $\pi_1(\Sigma)$.
In order to have rho invariants of $X_1 \cup_\Sigma X_0$ and $-X_0\cup_\Sigma X_2$ to be compared to $\rho_\alpha(X_1\cup_\Sigma X_2)$, we need to extend $\alpha$ to the fundamental groups of these manifolds. 
 This is possible if and only if we can construct an extension of $\alpha$ from $\pi_1(\Sigma)$ to $\pi_1(X_0)$: then, using Seifert--Van Kampen's theorem, this will be patched with $\alpha\colon \pi_1(X_1)\to U(n)$ to produce representations 
 $\pi_1(X_1 \cup_\Sigma X_0)\to U(n)$ and $\pi_1(-X_0\cup_\Sigma X_2)\to U(n)$. For simplicity, we will use the same notation $\alpha$ for all of these representations.
 Then, we want to compute the correction term $C$ in the formula 
	\begin{equation}\label{eq4}
	\begin{split}
	\rho_\alpha(X_1\cup_\Sigma X_2) = \rho_\alpha(X_1 \cup_\Sigma X_0) + \rho_\alpha(-X_0\cup_\Sigma X_2) + C.
	\end{split}
	\end{equation}
	Now, if  $X_1\cup_\Sigma X_0$ and $-X_0\cup_\Sigma X_2$ bound manifolds $W_1$ and $W_2$ such that the representation extends, then Wall's theorem, together with the Atiyah-Patodi-Singer signature theorem, tells us how to compute the correction term. Namely, in that case we have
	\begin{equation}\label{eqintroC}
	C=\tau(V^\alpha_{X_0},V_{X_1}^\alpha,V^\alpha_{X_2}) - n\, \tau(V_{X_0},V_{X_1},V_{X_2}).
	\end{equation}
The content of the main result of this section is that the correction term $C$ of \eqref{eq4} is always given by \eqref{eqintroC}, no matter whether the manifolds $W_1$ and $W_2$ exist. This should be compared with an analogous result for the untwisted eta invariant \cite[2.5]{bunke}.

%

\begin{theorem}\label{theocutandpaste}
	Let $X_1$, $X_2$ and $X_0$ be compact, oriented manifolds of dimension $2k-1$ with $\partial X_1=-\partial X_0=-\partial X_2$, and let $\alpha\colon \pi_1(X_1\cup_\partial X_2)\to U(n)$ be a representation that extends to $\pi_1(X_0)$. Then, for every choice of such an extension, we have
	\[
	\rho_\alpha(X_1\cup_\Sigma X_2) = \rho_\alpha(X_1 \cup_\Sigma X_0) + \rho_\alpha(-X_0\cup_\Sigma X_2)+
	C,
	\]
	where 
	\[C=\tau(V^\alpha_{X_0},V_{X_1}^\alpha,V^\alpha_{X_2}) - n\, \tau(V_{X_0},V_{X_1},V_{X_2}),\]
	with the first Maslov triple index  performed on $H_{k-1}(\partial X_1;\C^n_\alpha)$, and the second on $H_{k-1}(\partial X_1;\C)$.
\end{theorem}
\begin{proof}
Consider the oriented manifolds 
\[M_1:=[0,1]\times (X_1 \cup_\Sigma X_0),\quad M_2:=[0,1]\times (-X_0\cup_\Sigma X_2).\]
 We glue then $M_1$ with $M_2$ along $\{1\}\times X_0$, obtaining a topological oriented manifold $M$ to which $\alpha$ extends. 
 	 	
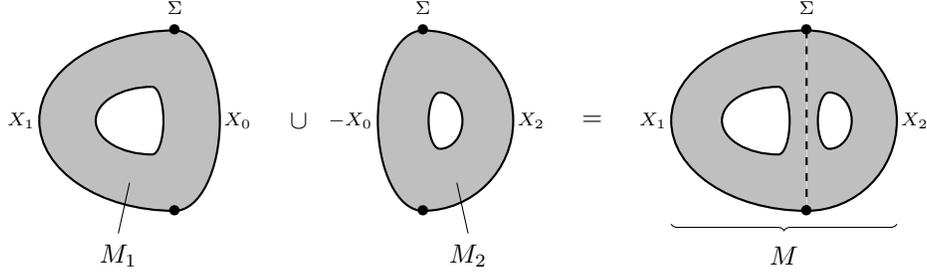
\begin{center}
 	 \begin{tikzpicture}[scale=0.6]
 	 \node at (0,1) {};
 	  \fill [lightgray] 
 	  (0,0) arc (90:270: 3 cm and 2cm)
 	  (0,-4) arc (-90:90: 1cm and 2cm);
 	 \fill[white] 
          (-0.5,-1.25) arc (90:270:1.25cm and 0.75cm)
 	  (-0.5,-2.75) arc (-90:90:0.25cm and 0.75cm);
 	 \draw[thick] (0,0) arc (90:270: 3 cm and 2cm);
 	 \draw[thick] (0,0) arc (90:-90: 1cm and 2cm);
 	 	 \draw[thick] (-0.5,-1.25) arc (90:270:1.25cm and 0.75cm);
 	  	 \draw[thick] (-0.5,-1.25) arc (90:-90:0.25cm and 0.75cm);
\draw[thin] (-1,-3.375) -- (-1.25, -4.5);
 \node at (-1.25,-5) {$M_1$};
 	 \fill [lightgray] 
 	  	  (5.5,0)  arc (90:270:1 cm and 2cm)
 	  	  (5.5,-4) arc (-90:90: 2cm);
 	  	 \fill[white] 
 	           (5.875,-1.375)  arc (90:270:0.25 cm and 0.625cm)
 	  	 (5.875,-2.625) arc (-90:90: 0.5cm and 0.625cm);
 	 	 \draw[thick] (5.5,0)  arc (90:270:1 cm and 2cm);
 	 \draw[thick] (5.5,-4) arc (-90:90: 2cm);
	 \draw[thick] (5.875,-1.375)  arc (90:270:0.25 cm and 0.625cm);
 	 \draw[thick] (5.875,-1.375) arc (90:-90: 0.5cm and 0.625cm);
 	 \draw[thin] (6.25,-3.375) -- (6.5, -4.5);
 	  \node at (6.5,-5) {$M_2$};
 	 \node at (0,0) {$\bullet$};
 	 \node at (0,0.5) {$\scriptstyle{\Sigma}$};
 	 \node at (0,-4) {$\bullet$};
 	  \node at (5.5,0) {$\bullet$};
 	  	 \node at (5.5,0.5) {$\scriptstyle{\Sigma}$};
 	 \node at (5.5,-4) {$\bullet$};
 	
 	 	    \node at (-3.4,-2) {$\scriptstyle{X_1}$};
 	 	    \node at (1.4,-2) {$\scriptstyle{X_0}$};
 	 	     \node at (2.75,-2) {$\cup$};
 	 	      \node at (3.9,-2) {$\scriptstyle{-X_0}$};
 	 	       \node at (7.9,-2) {$\scriptstyle{X_2}$}; 
 	 	        \node at (9.25,-2) {$=$};
 	 	        
 	 	         \fill [lightgray] 
 	 	         	  	 (14,0) arc (90:-90: 2cm)
 	 	         	  	   (14,0)  arc (90:270:3 cm and 2cm);
 	 	         	  	 \fill[white] 
 	 	         	           (13.375,-1.25) arc (90:270:1.25cm and 0.75cm)
 	 	         	  	 (13.375,-2.75) arc (-90:90:0.25cm and 0.75cm);
 	 	         \fill[white] 
 	 	         (14.5,-1.375)  arc (90:270:0.25 cm and 0.625cm)
 	 	         (14.5,-2.625) arc (-90:90: 0.5cm and 0.625cm);
 	 	 \draw[thick] (14,0) arc (90:-90: 2cm);
 	 	 \draw[thick] (14,0)  arc (90:270:3 cm and 2cm);
 	 	 	 \draw[thick] (13.375,-1.25) arc (90:270:1.25cm and 0.75cm);
 	 \draw[thick] (13.375,-1.25) arc (90:-90:0.25cm and 0.75cm);
 	  \draw[thick] (14.5,-1.375)  arc (90:270:0.25 cm and 0.625cm);
   	 \draw[thick] (14.5,-1.375) arc (90:-90: 0.5cm and 0.625cm);
 	 \draw[decoration={mirror, brace,raise=5pt},decorate]
 	  	  	 	   (11,-4) --  (16, -4);
 	  	  	 	   	  \node at (13.5,-5) {$M$};
 	 	        	 	 \node at (14,0) {$\bullet$};
 	 	        	 	 	 \node at (14,0.5) {$\scriptstyle{\Sigma}$};
 	 	        	 	 \node at (14,-4) {$\bullet$};
 	 	   \draw[thick, dashed] (14,0) -- (14,-4);
 	 	        	 	  \node at (10.6,-2) {$\scriptstyle{X_1}$};
 	 	        	 	    \node at (16.4,-2) {$\scriptstyle{X_2}$};
 	 	        	 	    
 	 \end{tikzpicture}
 	 	\end{center}
 The boundary of $M$ can be described topologically as
 \begin{equation}\label{eqboundaryM}
 \partial M=(-(X_1 \cup_\Sigma X_0)\sqcup -(-X_0\cup_\Sigma X_2)  )  \sqcup (X_1\cup_\Sigma X_2 ),\end{equation}
 and we can equip $M$ with a smooth structure such that \eqref{eqboundaryM} is satisfied in the smooth sense \cite[15.10.3]{tomdieck}. 
Thanks to Theorem~\ref{theoapsrho} (ii), we have then
  \begin{equation}\label{eqapsM}
  \rho_\alpha(\partial M)=n\,\sigma(M)-\sigma_\alpha(M).
  \end{equation}
By \eqref{eqboundaryM}, the left-hand term is given by
 \begin{equation}\label{eqrhoboundaryM}
  \rho_\alpha(\partial M)=-\rho_\alpha(X_1 \cup_\Sigma X_0) -\rho_\alpha(-X_0\cup_\Sigma X_2) +\rho_\alpha (X_1\cup_\Sigma X_2 ).\end{equation}
By Wall's non-additivity (Theorem~\ref{theowall}),  we can compute the twisted and untwisted signature of $M$  as
\[ 
 \begin{split}
  \sigma(M)&=\sigma(M_1)+\sigma(M_2) -\tau(V_{X_0},V_{X_1},V_{X_2}),\\
 \sigma_\alpha(M)&=\sigma_\alpha(M_1)+\sigma_\alpha(M_2) -\tau(V^\alpha_{X_0},V_{X_1}^\alpha,V^\alpha_{X_2}).
 \end{split}
\]
Now, all signatures of $M_1$ and $M_2$ vanish, as the reflection along the interval factor gives an orientation-reversing self-diffeomorphism which does not affect the representation. We obtain thus
\begin{equation}\label{eqwallM}
  \sigma(M)=-\tau(V_{X_0},V_{X_1},V_{X_2}),\quad \quad 
 \sigma_\alpha(M)= -\tau(V^\alpha_{X_0},V_{X_1}^\alpha,V^\alpha_{X_2}).
\end{equation}
Substituting \eqref{eqrhoboundaryM} and \eqref{eqwallM} into \eqref{eqapsM}, we get the desired formula.
\end{proof}

\begin{remark}
An alternative approach to proving Theorem~\ref{theocutandpaste} would be by using gluing formulas for rho invariants for manifolds with boundary, as defined by Kirk and Lesch \cite{kl1,kl2} (the discussion in \cite[Section~8.3]{kl1} might be hinting in this direction). This approach allows to prove the cut-and-paste formula at the level of eta invariants, leading to a slightly stronger result, albeit at the cost of more sophisticated tools to be introduced. 
See \cite[Section~2.3.4]{et} for details about this point of view.
\end{remark}

%
%
%
%
%

\section{Signatures of links and rho invariants}
In Section~\ref{ssframings} we introduce rationally framed links and set up some notation and easy results.
In Section~\ref{sscolored}, we introduce the multivariable signatures of Cimasoni and Florens
and recall a $4$-dimensional description for them.
In Section~\ref{sstheoyl}, we prove Proposition~\ref{extheomultisign}, which reinterprets the multivariable signature of a colored link as the rho invariants of some closed $3$-manifold associated to the link.
In Section~\ref{ssintegral}, we prove Theorem~\ref{theointeger}, which is a formula relating the multivariable signature of a link with the rho invariant of the $3$-manifolds obtained by integer surgery on the link.
In Section~\ref{ssrational}, we prove Theorem~\ref{theorational}, which under some additional hypotheses does the same for rational surgery.

\begin{remark}
In this section, given a manifold $X$, we will only deal with $1$-dimensional unitary representations of $\pi_1(X)$. As $U(1)$ is an abelian group, every representation $\alpha'\colon \pi_1(X)\to U(1)$ factors through the abelianization $\ab\colon \pi_1(X)\to H_1(X;\Z)$, and we can thus focus on representations $\alpha\colon H_1(X;\Z)\to U(1)$. 
We will normally write $\rho_\alpha$ and $H_\alpha$ to denote rho invariants and twisted homology associated to the representation $\alpha'=\alpha\circ \ab$.
\end{remark}
\begin{remark}\label{remarknontrivial}
A straightforward computation in twisted homology leads to the following well-known fact: for the $2$-dimensional torus $T^2$, we have
	$H_*(T^2;\C^n_\alpha)=0$
	for all non-trivial representation $\alpha\colon H_1(T;\Z)\to U(1)$.
	In this section, we will always use Theorem~\ref{theocutandpaste} in the situation where $\Sigma$ is a disjoint union of $2$-dimensional tori and where the restriction of $\alpha$ to the first homology of these is non-trivial. As a consequence, the Maslov triple index in twisted homology will always be $0$.
\end{remark}

\subsection{Links and framings}\label{ssframings}

Let $L=K_1\cup \cdots \cup K_k$ be an oriented link in $S^3$ (from now on, just a \emph{link}). 
By removing from $S^3$ the interior of a closed tubular neighborhood $N(L)$, we get its \emph{link exterior} 
\[X_L:=S^3 \setminus \inte (N(L)).\]
The link exterior $X_L$ is a compact, oriented $3$-manifold, whose boundary is a union of tori: to each link component $K_i\subseteq L$, there corresponds a boundary component $T_i=-\partial (N(K_i))$ (this is the orientation coming from being part of the boundary of $X_L$, and it is the one we shall always consider).
The link component $K_i$ determines the following two elements:
\begin{itemize}
	\item the \emph{meridian} of $K_i$ is the only element  $\mu_i\in H_1(T_i;\Z)$ whose image in $H_1(N(K_i);\Z)$ is $0$ and such that $\lk (\mu_i, K_i)=1$;
	\item the \emph{standard longitude} of $K_i$ is the only element $\lambda_i^s\in H_1(T_i;\Z)$   whose image in $H_1(N(K_i);\Z)$  is homologous to $K_i$ and such that $\lk (\lambda_i^s, K_i)=0$.
\end{itemize}
For a knot $K$, we shall often just use the notation $T$, $\mu$ and $\lambda$ for the boundary torus, the meridian and the standard longitude.
Observe that the algebraic intersection of these two elements is given by
\begin{equation}\label{eqintersection-1}
\mu_i\cdot\lambda_i^s=-1.
\end{equation}
We shall often consider the images of the above elements into homology with rational or complex coefficients without changing their names.

\begin{definition}
	A \emph{rational framing} on a link $L=K_1\cup\cdots \cup K_k$ is a $k$-tuple of rational numbers $r=(r_1,\dotsc, r_k)$. The pair $(L,r)$ is called a \emph{rationally framed link}. If all $r_i$'s are integers, we say that $r$ is an \emph{integer framing}. The \emph{framed longitudes} of a rationally framed link $(L,r)$ are the elements
	\[\lambda_i:=\lambda_i^s + r_i \mu_i\in H_1(T_i;\Q).\]
\end{definition}
Without changing notation for them, we shall now consider the images of the meridians and framed longitudes in the homology of the link complement $X_L$. It is an elementary well know fact that $H_1(X_L;\Z)$ is a free $\Z$-module generated by the meridians, and that each standard longitude satisfies
\begin{equation}\label{eqstandardlong}
\lambda_i^s = \sum_{j\ne i}\lk (K_i,K_j)\, \mu_j \in H_1(X_L;\Z).
\end{equation}
Next, we define the following matrix associated to a rationally framed link. 
\begin{definition}
	The \emph{framed linking matrix} of a framed $k$-component link $(L,r)$ is the symmetric matrix $\Lambda_r=(\Lambda_{ij})_{i,j}\in\Q^{k\times k}$ defined by
	\[\Lambda_{ij}=
	\begin{cases} 
	\lk (K_i, K_j), \quad &\text{if }i\ne j,\\
	r_i, \quad &\text{if } i=j.
	\end{cases}\]
\end{definition}
\begin{example}\label{exseifert}
	The \emph{Seifert framing} on a link $L=K_1\cup \cdots \cup K_k$ is the integer framing $(f_1, \cdots , f_k)$ defined by
	\[f_i:=-\sum_{j\ne i} \lk(K_i,K_j).\]
	In particular, the coefficients of its framed linking matrix satisfy
	\[\Lambda_{ii}=-\sum_{j\ne i}\Lambda_{ij}.\]
	 The framed longitudes $\lambda_i=\lambda_i^s + f_i\mu_i$ associated to the Seifert framing correspond to the intersections of a Seifert surface with the boundary tori $T_i$.
\end{example}
From \eqref{eqstandardlong}, together with the definition of the framed longitudes and of the framed linking matrix, it follows now immediately that the framed longitudes in the rational homology of $X_L$ are equal to
\begin{equation}\label{eqframedlong}
\lambda_i = \sum_{j=1}^k\Lambda_{ij}\, \mu_j \in H_1(X_L;\Q).
\end{equation}

For the computations involving the Maslov triple index, a good understanding of the complex symplectic space $H_1(\partial X_L;\C)$ is required. The next result summarizes some easy facts that we will need later.
\begin{lemma}\label{lemmasymplecticlink}
	Let $(L,r)$ be a $k$-component rationally framed link. Then: 
	\begin{enumerate}[(i)]
		\item the collection $\{\mu_1, \dotsc \mu_k, \, \lambda_1, \dotsc \lambda_k\}$ forms a basis for $H_1(\partial X_L;\C)$ which satisfies 
		\[
		\begin{cases}
		\mu_i\cdot\mu_j=\lambda_i\cdot\lambda_j=0,\\
		\mu_i\cdot \lambda_j = -\delta_{ij}
		\end{cases}
		\quad \text{for all }i,j;
		\]
		\item the canonical Lagrangian $V_{X_L}$ can be described explicitly as
		\[	V_{X_L}= \Span_\C\{v_1, \dotsc, v_k \}, \quad \text{where  } v_i=\lambda_i - \sum_{s=1}^k\Lambda_{is} \mu_s;\]
		\item the subspaces $\MM:=\{\mu_1,\dotsc , \mu_k\}$  and $\LL_r=\Span_\C\{\lambda_1,\dotsc , \lambda_k\}$ are Lagrangians, and their trisple Maslov index with the canonical Lagrangian is given by
		\[\tau(\MM, \LL_r, \VV_{X_L})=\sign \Lambda_r\]
	\end{enumerate}
\end{lemma}
\begin{proof}
	(i) is an immediate consequence of the definition of the framed meridians together with \eqref{eqintersection-1}. (ii) is an immediate consequence of \eqref{eqframedlong}. The fact that $\MM$ and $\LL_r$ are Lagrangians is obvious from $(i)$. We compute the Maslov triple index using the definition. As $\MM$ and $\LL_r$ are transverse, we have 
	$(\MM + \LL_r)\cap \VV_{X_L} = \VV_{X_L}$, and every generator $v_i$ can be written in a unique, obvious way as the sum of an element in $\MM$ and one in $\LL_r$. By Definition \ref{defmaslov}, then, $\tau(\MM, \LL_r, \VV_{X_L})$ is the signature of the Hermitian form
	$\psi\colon V_{X_L}\times V_{X_L}\to \C$ defined on the basis elements of $V_{X_L}$ by
	\[\psi (v_i, v_j)= \Big(- \sum_{s=1}^k\Lambda_{is} \mu_s \Big)\cdot v_j= -\sum_{s=1}^k \Lambda_{is} (\mu_s\cdot \lambda_j) =\Lambda_{ij}. \]
	It follows that $\tau(\MM, \LL_r, \VV_{X_L})= \sign \Lambda_{f_L}$, and (iii) is also proved.
\end{proof}

Given a rationally framed link $(L,r)$, we can consider the closed manifold $S_L(r)$ obtained by the \emph{Dehn surgery} along the framing. This is done in the following way: for each link component $K_i$, we choose coprime integers $(p_i,q_i)$ such that $p_i/q_i=r_i$, and glue a solid torus $Y_i$ to $X_L$ along the boundary component $T_i$ in such a way that the meridian of the solid torus is identified with the element $p_i\mu_i+q_i\lambda_i^s\in H_1(T_i;\Z)$.
In particular, $H_1(S_L(r);\Z)$ can be described as a quotient of $H_1(X_L;\Z)$.
\begin{definition}
	Given a link $L$ with a representation $\alpha\colon H_1(X_L;\Z)\to U(1)$,  we say that a rational framing $r$ on $L$ is \emph{compatible} with $\alpha$ if $\alpha$ factors through  $H_1(S_L(r);\Z)$.
\end{definition}
\begin{remark}\label{remcompatible}
From the definition of surgery, it is clear that a rational framing $r=(r_1,\dotsc , r_k)$ with $r_i=p_i/q_i$ as above is compatible with $\alpha$ if and only if $\alpha(p_i\mu_i+q_i\lambda_i^s)=1$ for all $i$.
Using \eqref{eqstandardlong}, we see that in terms of the coefficients of the framed linking matrix we have
\[
f \text{ is compatible with } \alpha 
\iff \prod_{j=1}^k\alpha(\mu_j)^{q_i\Lambda_{ij}}=1 \:\:\:\:\forall i  .
\]
\end{remark}
Note that, in general, given a representation $\alpha$, there might be no rational framing that is compatible with it.

     \subsection{Colored links and signatures}\label{sscolored}
A \emph{$n$-coloring} on a link $L$, for $n\in \N$, is a partition of its components into $n$ non-empty sublinks. Given a $k$-component link $L=K_1\cup \cdots \cup K_k$, we identify the coloring with a surjective function $c\colon \{1, \dotsc , k\}\to \{1, \dotsc , n\}$. The latter is the set of colors, and for every $1\leq s\leq n$ we define 
\[L_s:= \lmcup{c(j)=s}{} K_j\]
to be the sublink \emph{of color $s$}. The pair $(L,c)$ is called an \emph{$n$-colored} link. We shall systematically omit $c$ (which has to be considered as fixed) and simply call $L$ an $n$-colored link.

\begin{notation}
	For $n\in\N$, set $\T^n:=(U(1))^n$ and $\T_*^n:=(U(1)\setminus\{1\})^n$.
\end{notation}
     Using a generalization of the concept of Seifert surfaces (called \emph{C-complexes}), Cimasoni and Florens defined a multivariable version of the Levine-Tristram signature of a link \cite{cf}. 
  Given an $n$-colored link $L$, their multivariable signature is a function
     \[\sigma_L\colon \T_*^n \to \Z,\]
    which coincides for $n=1$ with the Levine-Tristram signature function.
We recall now the following $4$-dimensional description of the multivariable signature. We first need a definition.
     \begin{definition}
     A \emph{bounding surface} for an $n$-colored link $L$ is a union $F = F_1\cup \cdots \cup F_n$ of
     properly embedded, locally flat, compact, oriented surfaces~$F_i \subseteq D^4$
     with $\partial F_i = L_i\in \partial D^4=S^3$ and that only intersect each other transversally in double points.
     \end{definition}
 A bounding surface for $L$ can be obtained for example by pushing the interior of a C-complex into the interior of $D^4$ (see e.g.\ \cite[Section~3]{cft} for details).
     Given a bounding surface $F=F_1\cup \cdots \cup F_n\subseteq D^4$ for a link $L$, we can take a small tubular neighborhood $N(F_i)$ of each surface $F_i$ and define the \emph{exterior} of $F$ in $D^4$ as the $4$-manifold with boundary
     \[W_F:=D^4\sm (N(F_1)\cup\cdots \cup N(F_n)).\]
     It is easy to show that $H_1(W_F;\Z)$ is freely generated by the meridians of the surfaces $F_1, \dotsc, F_n$ (the meridian of $F_i$ being the image in $H_1(W_F;\Z)$ of any of the meridians of $L_i$). The following description of the Cimsoni-Florens signature is known \cite[Proposition~3.5]{cnt}.
     \begin{proposition}\label{propmultisign}
     Let $L$ be a an $n$-colored link in $S^3$ and let $F=F_1\cup \cdots \cup \F_n$ be a bounding surface for $L$. Let $\omega=(\omega_1, \dotsc, \omega_n)\in \T_*^n$, and let $\alpha\colon H_1(W_F;\Z)\to U(1)$ be the representation that sends the meridian of $F_i$ to $\omega_i$.  Then, we have
     \[\sigma_L(\omega)= \sigma_\alpha(W_F). \]
     \end{proposition}
We conclude this part with a couple of definitions.
\begin{definition}\label{defcolseifert}
	Let $L=K_1\cup\cdots \cup K_k$ be an $n$-colored link. 
	The \emph{colored Seifert framing} on $L$ is the integer framing $f_L=(f_1,\dotsc , f_k)$ given by 
	\[f_i:=- \!\!\!\sum_{\substack{j\ne i \text{ s.t.}\\ c(j)=c(i)}}\!\!\!\lk(K_i,K_j).\]
	In other words, it is the framing obtained by providing each colored sublink $L_i$ with its Seifert framing (see Example \ref{exseifert}). 
\end{definition}
Observe that the colored Seifert framing on $L$ needs not coincide with the Seifert framing of the underlying link. In fact, its framed longitudes correspond to the intersection of $\partial X_L$ with a C-complex. From the definition of the colored Seifert framing, it is immediate to see that the coefficients of the associated framed linking matrix satisfy
\begin{equation}\label{eqcolseifert}
\forall \: 1\leq i \leq k: \quad  \!\! \sum_{\substack{j\text{ s.t.}\\ c(j)=c(i)}}\!\!\!\Lambda_{ij}=0
\end{equation}

\begin{definition}
	Given an $n$-colored link $L$, a representation $\alpha\colon H_1(X_L;\Z)\to U(1)$ is said to be \emph{colored} if it sends meridians of the same color to the same value, i.e. if 
	\[c(i)=c(j) \implies \alpha(\mu_i)= \alpha(\mu_j).\]
	Given an element $\omega=(\omega_1, \dotsc , \omega_n)\in \T^n$, the colored representation $\alpha$ defined by $\alpha(\mu_i):=\omega_{c(i)}$ is said to be the representation \emph{associated} to $\omega$.
\end{definition}
  It is immediate to see that the above association gives a natural bijection between elements $\omega=(\omega_1, \dotsc , \omega_n)\in \T^n$ and colored representations $\alpha\colon H_1(X_L;\Z)\to U(1)$.

\subsection{Multivariable signatures as rho invariants}\label{sstheoyl}
	As the untwisted signature of $W_F$ is $0$ (see for example \cite[Proposition~3.3]{cnt}), Proposition~\ref{propmultisign} is equivalent to the formula
	\[\sigma_L(\omega)= \sigma_\alpha(W_F) - \sigma(W_F).\]
	In particular, by the Atiyah-Patodi-Singer theorem, we have
	\begin{equation}\label{eqsignaturedefect}
	\sigma_L(\omega)= -\rho_\alpha(\partial W_F).
	\end{equation}
We will give now a more explicit description of $\partial W_F$, and see how to replace it with a manifold which is independent of the choice of $F$.

We start by recalling the following construction, which is a special case of plumbing. 
Let $\Gamma$ be a graph whose set of vertices is $\{1,\dotsc , n\}$ and such that:
\begin{itemize}
	\item each vertex $i$ is decorated by a compact, oriented surface $\Sigma_i$ (or, equivalently, by pair of natural numbers $[g_i,r_i]$ corresponding to the genus and number of boundary components of $\Sigma_i$);
	\item each edge is decorated by a number $\varepsilon =\pm 1$.
\end{itemize}
 In the following, we shall refer to a graph with the decorations described above as a \emph{plumbing graph}.
We construct then an oriented $3$-manifold $P_\Gamma$ by the following process.
     \begin{enumerate}
     \item For each edge with endpoints $i$ and $j$, remove a small open disk from $\Sigma_i$ and one from $\Sigma_j$.  Let $\Sigma_1', \dotsc , \Sigma_j'$ be the resulting surfaces.
  \item For each edge with endpoints $i$ and $j$ and decoration $\varepsilon=\pm1$, glue the $3$-manifolds $\Sigma_i'\times S^1$ and $\Sigma_j'\times S^1$ along the boundary components coming from the two corresponding removed disks, according to the diffeomorphism
  $\varphi \colon  S^1\times S^1 \to   S^1\times S^1 $
   given by $\varphi (x,y)= (y^{\varepsilon},x^{\varepsilon})$.
     \item Set 
     \[ P_\Gamma:=\Big(\bigsqcup_{i=1}^n \Sigma_i'\times S^1 \Big)\,/\!\sim,  \]
     where $\sim$ is the equivalence relation given by the above gluings.
     \end{enumerate}
This construction coincides with that of \cite[Section~4]{cnt}. With respect to the general plumbing construction \cite{neumann}, it corresponds to the special case of all Euler numbers equal to $0$, and our definition of plumbing graph also reflects this specialization. The boundary of $P_\Gamma$ is a disjoint union of $r=\sum_i r_i$ tori. These tori maintain a preferred product structure, and we can describe the boundary of $P_\Gamma$ as
\begin{equation}\label{eqboundaryplumbing}
\partial P_\Gamma = \bigsqcup_{s=1}^n\partial \Sigma_i \times S^1.
\end{equation}
As in \cite{cnt}, we give the following definition.
 \begin{definition}
Given a plumbing graph $\Gamma$, the \emph{total weight} of a pair of vertices $\{s,t\}$, denoted by $p_\Gamma(s,t)$, is the integer obtained as the sum of the $\pm 1$-decorations of all the vertices with endpoints $s$ and $t$. If all total weights of $\Gamma$ are $0$, the plumbing graph $\Gamma$ is called \emph{balanced}.
 \end{definition}

\begin{example}\label{explumbing}
	Let $F=F_1\cup\cdots \cup F_n\subseteq  D^4$ be a bounding surface for $L$.  The boundary of $W_F$ then is given up to orientation-preserving diffeomorphism as 
	\begin{equation}\label{eqexplumbing}
	\partial W_F=X_L\cup_\partial (-P_{\Gamma_F}),
	\end{equation}
	where $\Gamma_F$ is a graph whose vertices $\{1, \dotsc , n\}$ are decorated by the surfaces $F_i$'s, and whose edges correspond to the intersection points between them, with the sign of the intersection as decoration (see \cite[Example 4.12]{cnt}). Observe that the total weights of $\Gamma_F$ are given by
\[
	p_{\Gamma_F}(s,t)=F_s\cdot F_t = \lk (L_s,L_t).
\]
The boundary of $-P_{\Gamma_F}$, which can be described as in \eqref{eqboundaryplumbing}, is identified in the gluing \eqref{eqexplumbing} with $\partial X_L$ as follows. The boundary piece $-\partial F_s\times S^1$ is identified with the boundary tori of color $s$ in $\partial X_L$, in such a way that:
\begin{enumerate}
\item the classes of the $S^1$-factors are glued to the meridians;
\item the classes of the boundary components of $F_s$ are glued to the framed longitudes associated to the colored Seifert framing.
\end{enumerate}
\end{example} 
 
%

As we have seen in \eqref{eqsignaturedefect}, the multivariable signature of $L$ can be expressed up to sign as the rho invariant of $\partial W_F$. In turn, as it appears from Example \ref{explumbing}, $\partial W_F$ can be described as the union of $X_L$ with some plumbed $3$-manifold which depends on the choice of of a bounding surface for $L$. In the next construction, we will build a plumbing graph $\Gamma_L$ associated to $L$. The associated closed $3$-manifold $Y_L$, obtained by gluing $P_{\Gamma_L}$ with $X_L$, will play the role of $\partial W_F$ with the advantage of being unequivocally determined by the link $L$.
\begin{notation}
	Given a link $L$ in $S^3$, let $\abs{L}\in \N$ be its number of components.
\end{notation}
\begin{construction}\label{construction}
Given an $n$-colored link $L$, we construct the associated plumbing graph $\Gamma_L$ in the following way:
 \begin{itemize}
 	\item the set of vertices is $\{1,\dotsc n\}$, and the vertex $i$ is decorated by a genus-0 surface with $\abs{L_i}$ boundary components;
 	\item between each pair of distinct vertices $i,j$ there are exactly $\abs{\lk (L_i,L_j)}$ edges, and they are all decorated by $\varepsilon=\sign(\lk (L_i,L_j))$.
 \end{itemize}
In other words, we are plumbing spheres with the appropriate number of punctures along the smallest graph $\Gamma$ whose total weights satisfy the condition $p_{\Gamma}(i,j)= \lk(L_i,L_j)$. 
 Then, we form a closed $3$-manifold $Y_L$ as 
\[
Y_L:= X_L\cup_\partial(- P_{\Gamma_L}).
\]
The identification along the boundary is defined in the exact same way as the one arising in Example \ref{explumbing}, i.e. identifying the classes corresponding to the $S^1$-factors of $-\partial P_{\Gamma_L}$ with the meridians of $L$ of the appropriate color, and the classes corresponding to the boundary components of the punctured sphere with the longitudes associated to the colored Seifert framing.
\end{construction}


Using some of the ideas of \cite{cnt} together with the cut-and-paste formula of Section~2, we will now show that the multivariable signature can be written as the rho invariant of $Y_L$.

          \begin{proposition}\label{extheomultisign}
          Let $L$ be an $n$-colored link. Let $\omega\in \T_*^n$, and let $\alpha\colon H_1(X_L;\Z)\to U(1)$ be the associated colored representation. Then, $\alpha$ can be extended to a representation of $H_1(Y_L;\Z)$ and, for any choice of an extension, we have 
          \[ \rho_\alpha(Y_L) = - \sigma_L(\omega).\]
          \end{proposition}
 
     \begin{proof}
     Let $F=F_1\cup\cdots \cup F_n\subseteq  D^4$ be a bounding surface for $L$. Then, as $\alpha$ is a colored representation, it extends to a representation $H_1(W_F;\Z)\to U(1)$ that, for all color $s$, sends the meridian of $F_s$ to  $\omega_s\in U(1)\subseteq\{1\}$. By \eqref{eqsignaturedefect}, we have hence
  \begin{equation}  \label{eqcfcnt}
   \rho_\alpha(\partial W_F)= - \sigma_L(\omega). 
   \end{equation}
As we have seen in Example \ref{explumbing}, we have $\partial W_F=X_L\cup_\partial (-P_{\Gamma_F})$, where $\Gamma_F$ is a plumbing graph determined by $F$.
    We perform now cut-and-paste by replacing $P_{\Gamma_F}$ with $P_{\Gamma_L}$. Schematically, this is
    \begin{equation}\label{eqprooftheocf}
     X_L\cup_\partial (-P_{\Gamma_F}) \quad  \leadsto  \quad X_L\cup_\partial (-P_{\Gamma_L})  \:\:\sqcup \:\:P_{\Gamma_L} \cup_\partial(-P_{\Gamma_F}).
     \end{equation}
    The manifold  $P_{\Gamma_L}\cup_\partial (-P_{\Gamma_F})$ can be seen as the plumbing along the graph $\Gamma$ whose vertices $\{1, \dotsc , n\}$ are decorated by the closed surfaces $\Sigma_s\cup_\partial (-F_s)$ and whose set of edges is the union of all edges of $\Gamma_L$ and $\Gamma_F$, with the decorations of the edges of $\Gamma_F$ changed of sign. 
     In particular, \eqref{eqprooftheocf} can be rewritten as
    \begin{equation}\label{eqprooftheocf2}
    \partial W_F \quad  \leadsto  \quad Y_L  \:\:\sqcup \:\: P_\Gamma.
    \end{equation}
    Observe that, by construction, for each pair of vertices $\{s,t\}$ we have
    \begin{equation}\label{eqproofweights}
    p_{\Gamma_L}(s,t) =  p_{\Gamma_F}(s,t),
    \end{equation}
    so that 
 \[p_{\Gamma}(s,t)= p_{\Gamma_L}(s,t)- p_{\Gamma_F}(s,t)=0\]
 i.e. $\Gamma$ is balanced.
 
    We will now prove that the representation $\alpha$ can be extended to $H_1(Y_L;\Z)$. In the gluings, the boundary of $X_L$ is identified with the boundaries of $P_{\Gamma_F}$ and $P_{\Gamma_L}$, leading to natural maps 
    \begin{equation}\label{eqphipsi}
    \varphi\colon H_1(\partial X_L;\Z)\to H_1(P_{\Gamma_F};\Z) , \quad \psi\colon H_1(\partial X_L;\Z)\to H_1(P_{\Gamma_L};\Z) 
    \end{equation}
    induced by the inclusions.
Standard Mayer-Vietoris computations, together with the equality \eqref{eqproofweights},  show that $\ker \varphi$ and $\ker\psi$ coincide, as both are generated by the following elements (compare with \cite[Lemma~4.7]{cnt}):
\begin{enumerate}[(i)]
	\item the differences  $\mu_i-\mu_j$ with $c(i)=c(j)$;
	\item for each color $s$, the element 
	\[\sum_{c(i)=s}\!\!\lambda_i - \sum_{t\ne s}\lk(L_s,L_t)\mu_{j_t},\]
	where $\mu_{j_t}$ is any meridian of color $t$.
\end{enumerate}
As $\alpha$ extends to $H_1(W_F;\Z)$, it is also defined on $H_1(\partial W_F;\Z)=H_1(P_{\Gamma_F};\Z)$. In particular, $\alpha$ has to be trivial on $\ker \varphi$ (this can be verified using the explicit description of the generators). From the fact that $\ker \varphi=\ker \psi$, we see then that $\alpha$ also admits an extension to $H_1(P_{\Gamma_L};\Z)$, because $U(1)$ is divisible. 
%

Pick any extension $\alpha\colon H_1(P_{\Gamma_L};\Z)\to U(1)$ of $\alpha$.
    We can now apply the cut-and-paste formula of Theorem~\ref{theocutandpaste} to \eqref{eqprooftheocf2}, obtaining
     \begin{equation}\label{eqcfcutandpaste}
     \rho_\alpha(\partial W_F)=\rho_\alpha(Y_L) + \rho_\alpha(P_{\Gamma})- \tau( V_{P_{\Gamma_L}},V_{X_L},V_{P_{\Gamma_F}})\end{equation}
     (the Maslov triple index in twisted homology is $0$ because all $\omega_i$'s are non-trivial by assumption; see Remark~\ref{remarknontrivial}).
 As we have seen, the plumbing graph $\Gamma$ is balanced.
    As a consequence, we have $\rho_\alpha(P_{\Gamma})=0$
     by a computation of Conway, Nagel and the author \cite[Proposition~4.10]{cnt}.
The Lagrangians $V_{M_F}$ and $V_{P_{\Gamma_L}}$ are identified under the gluing to
\[V_{M_F}=\ker\varphi\otimes \C, \quad V_{P_\Gamma}=\ker \psi\otimes \C,\]
where $\varphi$ and $\psi$ are the maps defined in \eqref{eqphipsi}. From the fact that $\ker\varphi=\ker\psi$ it follows then that  $V_{M_F}=V_{P_\Gamma}$, and hence we have $\tau(V_{X_L},V_{P_{\Gamma_F}}, V_{P_{\Gamma_L}})=0$.
The equality \eqref{eqcfcutandpaste} gets thus rewritten as
\[\rho_\alpha(\partial W_F)=\rho_\alpha(Y_L).\]
Substituting this into \eqref{eqcfcnt}, the proof is complete.
     \end{proof}


\begin{remark}\label{remcoltocol}
	If $L$ is color-to-color algebraically split, then $Y_L$ has a very simple description. In fact, in this case $\Gamma_L$ is a graph with no edges, and thus the associated plumbed $3$-manifold is 
	\[P_{\Gamma_L}=\bigsqcup_{s=1}^n \Sigma_s\times S^1,\]
	where $\Sigma_s$ is a sphere with $\abs{L_s}$ punctures. The multivariable signature is then (up to sign) just the rho invariant of the closed $3$-manifold obtained by gluing these products $\Sigma_s\times S^1$ to the link exterior $X_L$. 
	This holds in particular if $L$ is $1$-colored. As a consequence, we can always express the Levine-Tristram signature of $L$ as
	\[\sigma_L(\omega)=-\rho_\alpha(X_L\cup_\partial (-\Sigma\times S^1)),\]
	where $\Sigma$ is a punctured sphere. In the case of a knot, this gives the well-known description of the Levine-Tristram signature as the rho invariant of the manifold obtained by $0$-framed surgery. In the next two sections, we will study the relationship between rho invariants and Dehn surgery in more generality.
\end{remark}

 \subsection{Integral surgery}\label{ssintegral}
 We will now study the value of the rho invariant of manifolds obtained by integer surgery on a link $L$.
 We start by recalling the following result of Cimasoni and Florens \cite[Theorem~6.7]{cf}.
 \begin{theorem}[Cimasoni-Florens]\label{theocf}
 	Let $L$ be a $k$-colored $k$-component link. Let $q\in \N$ a positive integer and let $n_1,\dotsc, n_k\in \{1,\dotsc , p-1\}$ be integers, each of which is coprime with $q$.  Let $\omega=(e^{2\pi in _1/q}, \dotsc , e^{2\pi i n_k/q})\in (S^1\setminus\{1\})^n$, and let $\alpha \colon H_1(X_L;\Z)\to U(1)$ be the associated colored representation. Let $g$ be a compatible integral framing on $L$. Then, we have
 	\[
 	\rho_\alpha(S_L(g))= -\sigma_L(\omega) +\sum_{i<j}\Lambda_{ij} +  \sign \Lambda_g  -\lmfrac{2}{q^2}\sum_{i=1}^k (q-n_i)n_j \Lambda_{ij}.
 	\]
 \end{theorem}
 \begin{remark}The result of Cimasoni and Florens was originally written in terms of the Casson-Gordon invariant of Section~\ref{sscassgord}. We have translated it into a result about the rho invariant by using Proposition~\ref{propcassgord}.
 	\end{remark}
 
\begin{remark}Observe that formulas about any coloring of $L$ can be extracted from Theorem \ref{theocf}, as the signature function associated to any coloring can be easily deduced from the one associated to the maximal coloring \cite[Proposition~2.5]{cf}.
	For the $1$-coloring, a result of Casson and Gordon \cite[Lemma~3.1]{cassongordon2} about the Levine-Tristram signature can be obtained in this way.
	\end{remark}
We would like to remove the restrictions on the values of $\omega$ in the statement of Theorem \ref{theocf}. In order to be able to do so, we will impose some restrictions about the linking numbers between components of different colors.
 
 \begin{definition}
 	Let $L=K_1\cup\cdots \cup K_k$ be an $n$-colored link. Then
 	\begin{enumerate}[(i)]
 		\item if $\lk(L_s,L_t)=0$ for all pairs $(s,t)$ of distinct colors, we say that $L$ is \emph{color-to-color algebraically split};
 \item if every link component $K_i$ satisfies $\lk(K_i, L_s)=0$ for all $s\ne c(i)$,
 	we say that $L$ is \emph{component-to-color algebraically split}.
 \end{enumerate}
 \end{definition}
Of course, being component-to-color algebraically split is a stronger condition than being color-to-color algebraically split. 
Any $1$-colored link is component-to-color algebraically split, and so is any link with vanishing linking numbers, no matter what coloring it is assigned. The main result of this section is the following, which holds in the case of color-to-color algebraically split links.

 \begin{theorem}\label{theointeger}
 	Let $L$ be an $n$-colored link which is color-to-color algebraically split. Let $\omega\in \T_*^n$, and let $\alpha \colon H_1(X_L;\Z)\to U(1)$ be the associated colored representation. Let $g$ be a compatible integral framing on $L$. Then, we have
 	\[
 	\rho_\alpha(S_L(g))= -\sigma_L(\omega) +  \sign \Lambda_g -  2\sum_{s=1}^n h_s\theta_s(1-\theta_s),
 	\]
 	where, for each color $s$, $\theta_s\in (0,1)$ is such that $\omega_s=e^{2\pi i \theta_s}$, and $h_s$ is the sum of all the coefficients of the framed linking matrix  of the sublink $L_s$.
 \end{theorem}
 \begin{proof}
 	Let $k$ be the number of components of $L$.
 	We perform the cut-and-paste illustrated schematically by 
 	\begin{equation}\label{eqprooftheointegerboh}
 	X_L\cup_\partial Y \quad \leadsto \quad  X_L\cup_\partial (-P_{\Gamma_L}) \:\:\sqcup \:\: P_{\Gamma_L}\cup_\partial Y, 
 	\end{equation}
 	where 
 	$Y=Y_1\sqcup \cdots \sqcup Y_k$
 	 is a disjoint union of $k$ solid tori, glued along the framing $g$, and $P_{\Gamma_L}$ is glued as prescribed by Construction \ref{construction}. 
 	As $L$ is color-to-color algebraically split, by Remark~\ref{remcoltocol} we have
 	\[P_{\Gamma_L}=\bigsqcup_{s=1}^n \Sigma_s\times S^1,\]
 	where, for each color $s$, $\Sigma_s$ is a sphere with $\abs{L_s}$ punctures. 
The manifold $P_{\Gamma_L}\cup_\partial Y$ is thus the disjoint union of the $n$ closed manifolds obtained by capping all of the $\Sigma_s\times S^1$'s appropriately with solid tori. 
\begin{claim} Up to an orientation-preserving diffeomorphism, we have
	\[P_{\Gamma_L}\cup_\partial Y= -\bigsqcup_{s=1}^n   L(h_s,1),\]
	in such a way that the element $1\in \Z/h_s=\pi_1(L(h_s,1))$ is given by the class $[S^1]\in H_1(\Sigma_s\times S^1;\Z)$.
\end{claim}
\noindent
We postpone for the moment the proof of this claim.
 	As a consequence of Claim~1, \eqref{eqprooftheointegerboh} can be rewritten up to orientation-preserving diffeomorphism as
 	\[S_L(f_L) \quad  \leadsto  \quad Y_L  \:\:\sqcup \:\: -\bigsqcup_{s=1}^n L(h_s,1). \]
 	Observe that the restriction of $\alpha$ to $H_1(\Sigma_s\times S^1)$ sends the class $[S^1]$ to $\omega_s$, as the boundary circles of the form $\{p_i\}\times S^1$ are identified with the meridians of the link. The representation $\alpha$ extends thus to $L(h_s,1)$ in such a way that, for each color $s$, the element $1\in \Z/h_s=H_1(L(h_s;\Z))$ is sent to $\omega_s$.
We can now apply Theorem~\ref{theocutandpaste} and get
 	\begin{equation}\label{eqproofnp}
 	\rho_\alpha(S_L(f_L)) =\rho_\alpha(Y_L)- \sum_{s=1}^n\rho (L(h_s,1),\omega_s) - \tau(V_P, V_{X_L}, V_Y) 
 	\end{equation}
 	(as usual, there is no Maslov triple index in twisted cohomology because of Remark~\ref{remarknontrivial}).
 	The first summand in the right-hand term of \eqref{eqproofnp} is minus the multivariable signature of $L$ thanks to Proposition~\ref{extheomultisign}. Using Corollary~\ref{corlensinteger} to describe the rho invariant of the lens spaces at hand, and swapping  the second and third variables of the Maslov triple index (see Proposition \ref{propmaslov} (i)), we can rewrite \eqref{eqproofnp} as 
 	\begin{equation}\label{eqproofnp'}
 \rho_\alpha(S_L(f_L)) =-\sigma_L(\omega)- \sum_{s=1}^n (2h_s\theta_s(1-\theta_s)- \sgn(h_s)) + \tau(V_P, V_Y, V_{X_L}). 
 	\end{equation}
 	Hence, in order to conclude, we need to identify the Maslov triple index term in \eqref{eqproofnp'}. The rest of the proof is devoted to this.
 	
 	Instead of trying to calculate the Maslov triple index directly, we use the cocycle property of our three Lagrangians together with $\MM:=\Span_\C\{\mu_1,\dotsc , \mu_k\}$ to simplify this task. Namely, using Proposition \ref{propmaslov} (ii), we find
 \begin{equation}\label{eqproofintegermaslov} 
 	\tau(V_P, V_Y, V_{X_L}) = \tau (\MM,V_P, V_Y)- \tau(\MM, V_P, V_{X_L}) + \tau(\MM, V_Y, V_{X_L}).
 	\end{equation}
 We set the following notation. Let $\lambda_1,\dotsc , \lambda_k$ be the framed longitudes associated to the surgery framing $g=(g_1,\dotsc, g_k)$, and let $\lambda_1,\dotsc , \lambda_k$ bet the framed longitudes associated to the colored Seifert framing $f_L=(f_1,\dotsc, f_k)$. By definition, then, we have 
 \begin{equation}\label{eqprooflambda'}
 \lambda_i'=\lambda_i + (f_i-g_i)\mu_i \quad \text{for all } i=1,\dotsc , k.
 \end{equation}
 We also use the notation $\Lambda_{ij}$ for the coefficients of the framed linking matrix $\Lambda_g$, and $\Lambda_{ij}'$ for those of the frame linking matrix $\Lambda_{f_L}$ (these two matrices only differ on the diagonal).
 We have the following description of the four Lagrangian subspaces appearing in the right-hand term of \eqref{eqproofintegermaslov}:
 	\begin{equation*}
 \begin{split}
 \MM&=\Span_\C\{\mu_1,\dotsc , \mu_k\},\\
 V_P&=\Span_\C\{\mu_i-\mu_j\, \:\vert \: c(i)=c(j)\} \oplus \Span_\C\{ v_1, \dotsc , v_n\}, \quad \text{where } v_s=\sum_{c(i)=s}\lambda_i',\\
 V_Y&=\Span_\C\{\lambda_1,\dotsc, \lambda_k\}
 ,\\
 V_{X_L}&= \Span_\C\{w_1, \dotsc, w_k\}, \quad \text{where } w_i=\lambda_i - \sum_{j=1}^k\Lambda_{ij} \mu_j= \lambda_i' - \sum_{j=1}^k\Lambda_{ij}' \mu_j.
 \end{split}
 \end{equation*}
 We compute now the three summands separately. We will prove the following.
 	\begin{claim}
 		$\tau (\MM,V_P, V_Y)=-\sum_{s=1}^n\sgn(h_s)$.
 		\end{claim}
  	 	\begin{claim}
 	$\tau(\MM, V_P, V_{X_L})=0$.
 \end{claim}
 	 	\begin{claim}
 		$\tau(\MM, V_Y, V_{X_L})=\sign(\Lambda_g)$.
 	\end{claim}
 
These three claims, together with \eqref{eqproofnp'} and \eqref{eqproofintegermaslov}, lead to the desired formula. In order to conclude, we are only left with proving Claim 1 to 4.
 	
\begin{claimproof}2.
 	Write 	$\tau (\MM,V_P, V_Y)=\tau (V_Y,\MM,V_P)$. Clearly, we have 
 	\[(V_Y +\MM)\cap V_P=V_P.\] 
 	In fact, the generators of $V_P$ of the form $\mu_i-\mu_j$ are in $\MM$, and the generators $v_s$  
 	can be written as
 	\begin{equation}\label{eqproofclaim}
 	v_s= \sum_{c(i)=s}\lambda_i + \sum_{c(i)=s}(f_i-g_i)\mu_i,
 	\end{equation}
 	where the first summand is in $V_Y$, and the second summand is in $\MM$. 
 	Let 
 	\[\psi\colon V_P\times V_P \to \C\]
 	 be the Hermitian form associated to the triple $(V_Y,\MM,V_P)$, whose signature is $\tau (V_Y,\MM,V_P)$. The generators of the form $\mu_i-\mu_j$ are clearly in the radical of $\psi$, as they belong to the Lagrangian $\MM$. As a consequence, it is enough to study $\psi$ on the span of $v_1, \dotsc , v_n$. Using the definition of $\psi$ and the decomposition \eqref{eqproofclaim}, we compute
 	\[\psi(v_s,v_t)=
 	\Big(\sum_{c(i)=s}\lambda_i\Big)\cdot \Big(\sum_{c(j)=t}(f_j-g_j)\mu_j\Big)=
 	\begin{cases}
 	\displaystyle \sum_{c(i)=s}(f_i-g_i), &\text{if } s=t,\\
 	0, &\text{otherwise}.
 	\end{cases}\]
 	As a consequence, we have 
 	\[\sign \psi = \sum_{s=1}^n\sgn \Big(\sum_{c(i)=s}(f_i-g_i) \Big).\]
 	By definition of the colored Seifert framing, together with the fact that $g_i=\Lambda_{ii}$, we can now compute that, for each color $s$, we have 
 	\[\sum_{c(i)=s}(f_i-g_i)=\sum_{c(i)=s}\Big(-\sum_{\substack{c(j)=s\\j\ne i}}\lk(K_i,K_j)\Big)-\sum_{c(i)=s}g_i= - \!\!\!\!\!\!\sum_{c(i)=c(j)=s} \!\!\!\!\!\! \Lambda_{ij}=-h_s.\]
Putting these computations together, we find the equation in the statement of the claim.
\end{claimproof}

 \begin{claimproof}3.
 	The space $(\MM+V_P)\cap V_{X_L}$ is the $n$-dimensional subspace generated by the terms
\begin{equation}\label{eqproofclaim2}
 		z_s:=\sum_{c(i)=s}w_i = -\sum_{c(i)=s}\sum_{j=1}^k\Lambda_{ij}'\mu_j + \sum _{c(i)=s}\lambda_i'
 		\end{equation}
 where the first summand is in $\MM$ and the second summand is in $V_P$.
 Let 
 \[\varphi \colon \Span_\C\{z_1, \dotsc , z_n\} \times \Span_\C\{z_1, \dotsc , z_n\} \to \C\]
be the Hermitian form associated to the triple $(\MM, V_P, V_{X_L})$. Then, from the decomposition \eqref{eqproofclaim2} we can compute
\[\varphi(z_s,z_t)= \Big( -\sum_{c(i)=s}\sum_{j=1}^k\Lambda_{ij}'\mu_j \Big)\cdot \Big(\sum _{c(i)=t}\lambda_i'\Big) =-\sum_{\substack{c(i)=s\\c(j)=t}} \Lambda_{ij}'.
\]
For $s=t$, this is $0$ by definition of the colored Seifert framing. For $s\ne t$, instead, it is equal to $\lk(L_s,L_t)$, which is $0$ because the link is color-to-color algebraically split. 
In particular, the form $\varphi$ is trivial and the Maslov triple index is $0$ as claimed.
 \end{claimproof}
 	
\begin{claimproof}4.
	This follows immediately from Lemma~\ref{lemmasymplecticlink} (iii), as $V_Y=\LL_g$.
\end{claimproof}

\begin{claimproof}1.
	As we have observed, $P_{\Gamma_L}$ is the disjoint union of manifolds of the form $\Sigma_s\times S^1$, where $\Sigma_s$ is a punctured sphere.
	By construction, in the gluing $X_L\cup_\partial Y$, the meridian $m_i$ of the solid torus $Y_i$ is identified with the framed longitude $\lambda_i$ of the surgery framing $g$.
	On the other hand, when $-P_{\Gamma_L}$ is glued to $X_L$, a boundary component $C_i\times S^1\subseteq \partial \Sigma_s\times S^1$ is identified with the boundary torus $K_i$ in such a way that, homologically, the class $[C_i]$ coincides with the framed longitude $\lambda_i'$ of the colored Seifert framing, and
	$[S^1]$ coincide with the meridian $\mu_i$.
	 The  ``by-product'' gluing $P_{\Gamma_L}\cup_\partial Y$ is the result of capping the boundary $\Sigma_s\times S^1$ with solid tori.
	As a consequence of these and \eqref{eqprooflambda'}, these cappings are given by the identifications
	\[m_i = \lambda_i = \lambda_i'+(g_i-f_i)\mu_i =[C_i]+(g_i-f_i)[S^1]. \]
	In particular, for each $s$, this construction leads to the lens space $L(-h'_s, 1)$ or equivalently to $-L(h_s',1)$, with 
	\[h'_s:= \sum_{c(i)=s} (g_i-f_i).\]
	It is also easy to see verify that the element $1\in \Z/h_s'$ corresponds to $[S^1]$ as desired.
	The proof is concluded by proving that $h_s'=h_s$. This follows from the definition of the colored Seifert framing and the sequence of equalities
	\[\sum_{c(i)=s} (g_i-f_i)= \sum_{c(i)=s}g_i + \sum_{c(i)=s}\sum_{\substack{\:j\ne i \text{ s.t.\ }\\c(j)=s}} \lk(K_i,K_j)=\!\!\!\!\! \sum _{c(i)=c(j)=s}\!\!\!\!\!\Lambda_{ij}=h_s.\]
\end{claimproof}
 	\end{proof}

 \begin{remark}
 Applying Theorem~ \ref{theointeger} in the $1$-colored setting, where the first hypothesis is always satisfied, we get the following formula relating the rho invariant of the manifold obtained by surgery and the Levine-Tristram signature:
 	\begin{equation}\label{eqcg}
 	\rho_\alpha(S_L(g))= -\sigma_L(\omega) +  \sign \Lambda_g -  2 \Big(\sum_{i,j}\Lambda_{ij}\Big)\theta(1-\theta),
 	\end{equation}
where $\theta\in (0,1)$ is such that $\omega=e^{2\pi_i\theta}$ and the sum is over the coefficients of $\Lambda_g$. For $\theta\in \Q$, this coincides with a formula of Casson and Gordon \cite[Lemma~3.1]{cassongordon2}. 
 \end{remark}

\begin{example}
	Let $r,s$ be positive coprime integers, and let $T(r,s)$ denote the $(r,s)$-torus knot. 
	The	$(rs-1)$-Dehn surgery on $T(r,s)$ gives a manifold which is orientation-preserving diffeomorphic to the lens space $L(rs-1,s^2)$ \cite[Proposition~3.2]{moser}. 
	Let $\zeta:=e^{2\pi i/(rs-1)}$, and let $0\leq k \leq rs-2$.
	Keeping track of the induced map on the fundamental group under this diffeomorphism, by \eqref{eqcg} (applied with $\omega=\zeta^k$) we obtain
	\[\rho(L(rs-1, s^2), \zeta^{krs^2})= -\sigma_{T(r,s)}(\zeta^{k}) +1 -\lmfrac{2k(rs-1-k)}{rs-1}.\]
	It might be interesting to compare this formula to other known computations for the Levine-Tristram signature of torus knots (see e.g.\ the paper of Borodzik and Oleszkiewicz \cite{maciej}).
\end{example}

Suppose now that the $n$-colored link $L$ is component-to-color algebraically split. Observe that, under this assumption, the colored Seifert framing coincides with the usual Seifert framing, i.e.\ with the colored Seifert framing associated to the $1$-coloring of the same underlying link. As explained by the next result, this framing has the important property of being compatible with all $U(1)$-representations of $H_1(X_L;\Z)\to U(1)$.
 
\begin{lemma}\label{lemmaseifert}
Let $L$ be an $n$-colored link which is component-to-color algebraically split. Then, the colored Seifert framing $f_L$ is compatible with all colored representations $\alpha\colon H_1(X_L;\Z)\to U(1)$.
\end{lemma}
\begin{proof}
By 	Remark~\ref{remcompatible}, we need to prove that 
		\begin{equation}\label{eqprooflemmaseifert}
	\prod_{j=1}^k\alpha(\mu_j)^{\Lambda_{ij}}=1 \quad \quad \text{for all } i.
	\end{equation}
	Let $\omega\in \T^n$ be the element determined by the relations $\alpha(\mu_i)=\omega_{c(i)}$ for all $i$. 
We can write then	
\begin{equation}\label{eqprooflemmaseifert2}
\prod_{j=1}^k\alpha(\mu_j)^{\Lambda_{ij}}=\prod_{s=1}^n \omega_s^{\sum_{c(j)=s}\Lambda_{ij}}.
\end{equation}
	Since $L$ is component-to-color algebraically split, for every $s$ different from $s_i:=c(i)$ we have 
	\[\sum_{c(j)=s}\Lambda_{ij}=\lk(K_i,L_s)=0.\]
As a consequence, \eqref{eqprooflemmaseifert2} can be rewritten as 
\begin{equation}\label{eqprooflemmaseifert3}
\prod_{j=1}^k\alpha(\mu_j)^{\Lambda_{ij}}= \omega_{s_i}^{\sum_{c(j)=s_i}\Lambda_{ij}}.
\end{equation}
By \eqref{eqcolseifert}, moreover, 
the exponent in the right-hand term of \eqref{eqprooflemmaseifert3} is $0$, and thus \eqref{eqprooflemmaseifert} is satisfied. 
\end{proof}

\begin{corollary}\label{corseifert}
		Let $L$ be an $n$-colored link which is component-to-color algebraically split. Let $\omega\in \T_*^n$, and let $\alpha \colon H_1(X_L;\Z)\to U(1)$ be the associated colored representation. Then, $\alpha$ extends to $H_1(S_L(f_L);\Z)$ and we have
	\[
	\rho_\alpha(S_L(f_L))= -\sigma_L(\omega) +  \sign \Lambda_{f_L}.
	\]
\end{corollary}
\begin{proof}
The representation $\alpha$ extends to $H_1(S_L(f_L);\Z)$ thanks to Lemma~\ref{lemmaseifert}, so that we can apply Theorem~\ref{theointeger}. The desired formula follows then from the observation that, for the colored Seifert framing, thanks to \eqref{eqcolseifert} for any color $s$ we have
\[h_s\stackrel{\text{def}}{=}\sum_{c(i)=c(j)=s}\Lambda_{ij}=\sum_{c(i)=s} \sum_{c(j)=s}\Lambda_{ij}=\sum_{c(i)=s}0=0.\]
\end{proof}

\begin{remark}
	Both Lemma~\ref{lemmaseifert} and Corollary~\ref{corseifert} hold in particular in the $1$-colored setting, where they were proved by Nagel and Powell \cite[Section~5]{nagelpowell}. Our work is a generalization of this to the multivariable setting.
\end{remark}
 
 \subsection{Rational surgery}\label{ssrational}

 We now want to study the rho invariant of the closed manifold obtained by surgery along a rational framing on a link. We start from the case of a knot, as the statement and the proof are a bit simpler in this setting.
 Observe that, for a knot $K$, a representation $\psi\colon H_1(X_K)\to U(1)$, extends to $M_K(p/q)$ if and only if $\psi(\mu)^p=1$, i.e.\ if and only if $\omega:=\psi(\mu)$ is a $p$\textsuperscript{th} root of unity.
 
 \begin{proposition}\label{propdehnknot}
 	Let $K$ be a knot, let $\omega$ be a $p$\textsuperscript{th} root of unity, and let $\alpha\colon H_1(X_K;\Z)\to U(1)$ be the representation defined by $\alpha(\mu)=\omega$. Then, we have
 	\[\rho_\alpha(S_K(p/q))= -\sigma_K(\omega) - \rho(L(p,q),\omega).\]
 \end{proposition}
 \begin{proof}
 	We perform cut-and-paste on $S_K(p/q)$ in the following way: we cut out the solid torus $Y=D^2\times S^1$ of the filling, and we replace it with another copy $Y'=D^2\times S^1$, this time glued along the $0$ framing. 
 	It is convenient to actually glue $-Y'$ instead of $Y'$: in such a way, we can define an orientation-reversing diffeomorphism between $-\partial Y'$ and $\partial X_L$ which gives the identifications
 	\[m'=\lambda, \quad \quad l' =\mu\]
 	between the standard basis $(m',l')$ of $H_1(Y';\Z)$ and the basis $(\mu, \lambda)$ of $H_1(\partial X_L;\Z)$. On the other hand, the meridian $m$ of $Y$ is identified with $p\mu+q\lambda$.
 	Schematically, we write 
 	\begin{equation} \label{eqproofdehnknot}
 	X_K\cup_\partial Y   \quad  \leadsto  \quad X_K\cup_\partial (-Y') \:\:\sqcup \:\: Y'\cup_\partial Y.
 	\end{equation}
 	The union of solid tori $Y'\cup_\Sigma Y$ is now given along a diffeomorphism which gives the identification of $m$ with $qm'+pl'$, so that the resulting manifold is the lens space $L(p,-q)\cong -L(p,q)$ (see Section \ref{sslens}). In particular, \eqref{eqproofdehnknot} can be rewritten as 
 	\[S_K(p/q) \quad  \leadsto  \quad S_K(0)  \:\:\sqcup \:\:  -L(p,q).  \]
 	Observe that the generator $1\in \Z/p=\pi_1(L(p,q))$ corresponds by construction to the longitude $l'$. In turn, $l'$ is glued in the surgery with the meridian $\mu$ of $K$. As a consequence, the extension of the representation $\alpha$ to $H_1(L(p,q);\Z)$ is the representation $\Z/p\to U(1)$ given by $1\mapsto \omega$.
 	Before applying the cut-and-paste formula, we observe that the result is trivially true for $\omega=1$. We shall hence suppose $\omega\ne 1$. 
 	Thanks to Theorem~\ref{theocutandpaste}, we can now compute
 	\begin{equation}\label{eqcutpasteknot}
 	\rho_\alpha(S_K(p/q))= \rho_\alpha(S_K(0)) - \rho(L(p,q),\omega) -\tau(V_{Y'}, V_{X_K}, V_Y)
 	\end{equation}
 	(as usual, the Maslov triple index in twisted homology is $0$ thanks to Remark~\ref{remarknontrivial} and the assumption $\omega\ne 1$).
 	Now, we know that the first summand in the right-hand term of \eqref{eqcutpasteknot} is minus the Levine-Tristram signature (see Remark~\ref{remcoltocol}), while the second summand coincides with the one of the statement. As a consequence, to complete the proof it is enough to show that $\tau(V_{Y'}, V_{X_K}, V_Y)=0$. We are going to describe the Lagrangians explicitly. Considering the identifications given by the gluings, in terms of the basis $(\mu,\lambda)$ of $H_1(\partial X_K;\C)$ we have
 	\[V_{Y'}=\Span_\C(\lambda), \quad  V_{X_K}=\Span_\C(\lambda),\quad  V_Y=\Span_\C(p\mu +q\lambda).\]
 	As the first two subspaces coincide, the Maslov triple index is $0$, and the proof is complete.
 \end{proof}

 We shall now prove the general version of Proposition \ref{propdehnknot}.
 \begin{theorem} \label{theorational}
 Let $L$ be an $n$-colored, $k$-component link that is component-to-color algebraically split. Let $\omega\in \T_*^n$, and let $\alpha\colon\pi_1(X_L)\to U(1)$ be the associated colored representation. Let $r$ be a compatible rational framing on $L$. Then, we have
 	\[
 	\rho_\alpha(M_L(r))= -\sigma_L(\omega) +  \sign \Lambda_r - \sum_{i=1}^k (\rho(L(p_i,q_i),\omega_{c(i)}) + \sgn(p_i/q_i)) ,
 	\]
 	where $p_i,q_i$ are coprime integers such that $r_i-f_i=p_i/q_i$ (here $f_i$ is the $i$-th coefficient of the Seifert framing).
 \end{theorem}
 \begin{proof}
 	We generalize the cut-and-paste construction of Proposition~\ref{propdehnknot}, removing the union of solid tori $Y$ coming from the $r$-framed filling of $X_L$, and replacing them with a union of solid tori $Y'$ glued along the colored Seifert framing. As the difference between these framings is now given by the $k$-tuple $(p_1/q_i, \dotsc , p_k/q_k)$, the same argument used in the proof of Proposition~\ref{propdehnknot} (repeated now for each link component) implies that this cut-and-paste can be written as
 	\[S_L(r) \quad  \leadsto  \quad S_L(f_L)  \:\:\sqcup \:\: \Big(\bigsqcup_{i=1}^{k} -L(p_i,q_i) \Big).\] 
 	In particular, Theorem~\ref{theocutandpaste} gives in this case 
 	\[\rho_\alpha(M_L(r)) = \rho_\alpha( M_L(f_L)) - \sum_{i=1}^k \rho_\alpha(L(p_i,q_i))  - \tau(V_{Y'}, V_{X_L}, V_Y). \]
 	Applying Corollary~\ref{corseifert} and rearranging the variables of the Maslov triple index with Proposition \ref{propmaslov} (i), we can rewrite the last equation as 
 	\begin{equation}\label{eqprooftheolt}
 	\rho_\alpha(M_L(r)) = -\sigma_L(\omega) + \sign \Lambda_{f_L} - \sum_{i=1}^k \rho(L(p,q), \omega_{c(i)})  + \tau(V_{Y'}, V_Y, V_{X_L}).
 	\end{equation}
	The computation of the Maslov triple index is more involved than in the case of knots. 
 	As we did in the proof of Theorem~\ref{theointeger}, instead of trying to calculate it directly, we use the cocycle property of our three Lagrangians together with $\MM:=\{\mu_1,\dotsc , \mu_k\}$ to simplify this task. Namely, using Proposition \ref{propmaslov} (ii), we find
 	\begin{equation}\label{eqprooftheolt'} 
 	\tau(V_{Y'}, V_Y, V_{X_L}) = \tau (\MM,V_{Y'}, V_Y)- \tau(\MM, V_{Y'}, V_{X_L}) + \tau(\MM, V_Y, V_{X_L}).
 	\end{equation}
\begin{claim}
$	\tau (\MM,V_{Y'}, V_Y)= -\sum_{i=1}^k \sign(p_i/q_i)$.
\end{claim}
\begin{claim}
	$\tau(\MM, V_{Y'}, V_{X_L})= \sign\Lambda_{f_L}$ and $\tau(\MM, V_Y, V_{X_L})=\sign \Lambda_r$.	
\end{claim}
Thanks to \eqref{eqprooftheolt'}, the two claims (whose proof we postpone for the moment) combine to give
\[\tau(V_{Y'}, V_Y, V_{X_L}) = -\sum_{i=1}^k \sign(p_i/q_i)- \sign\Lambda_{f_L} + \sign \Lambda_r.\]
Substituting this value into \eqref{eqprooftheolt}, we obtain the formula in the statement of the theorem.
\begin{claimproof}1.
 	Let $\lambda_1,\dotsc , \lambda_k$ be the framed longitudes corresponding to the framing $r$, and let $\lambda_1',\dotsc , \lambda_k'$ be the framed longitudes of the Seifert framing. Then, it is clear that 
 	\[
 	V_Y=\Span_\C\{\lambda_1, \cdots , \lambda_k\}, \quad \quad V_{Y'}=\Span_\C\{\lambda_1', \cdots , \lambda_k'\}.
 	\]
 	Moreover, by definition we have $\lambda_i=\lambda_i^s+r_i\mu_i$ and $\lambda_i'=\lambda_i^s+f_i\mu_i$, so that the two sets of longitudes are related by 
 	\[\lambda_i=\lambda_i'+(r_i-f_i)\mu_i= \lambda_i'+ \lmfrac{p_i}{q_i} \mu_i.\]
 	In particular, the three Lagrangians in the first summand of  \eqref{eqprooftheolt'} all split according to the symplectic decomposition 
 	\[H_1(\partial X_L;\C)=\bigoplus_{i=1}^k \Span_\C\{\mu_i,\lambda_i\}.\]
 	As the pair $(\mu_i,\lambda_i')$ is a symplectic basis for $(H_1(T_i;\C), \cdot)$ (see Example \ref{extaudim2}), it is immediate to compute
 	\[ \tau (\MM,V_{Y'}, V_Y)= \sum_{i=1}^k \tau\Big(\mu_i, \lambda_i', \, \lambda_i' + \lmfrac{p_i}{q_i} \mu_i\Big)= -\sum_{i=1}^k \sign(p_i/q_i).\]
 \end{claimproof}
\begin{claimproof}2.
 The claim follows immediately from Lemma~\ref{lemmasymplecticlink} (iii), as we have $V_Y=\LL_r$ and $V_{Y'}=\LL_{f_L}$.
 	\end{claimproof}
 \end{proof}     
	The hypothesis of $L$ being component-to-color algebraically split is required in order to have a single nice explicit formula. This is because in this case we can apply the cut-and-paste formula after performing Dehn surgery along the Seifert framing, which is compatible with all representations thanks to Lemma~\ref{lemmaseifert}. However, even for links that are not component-to-color algebraically split, it is possible to deal with rational framings using Theorem~\ref{theocf} or Theorem~\ref{theointeger} as a starting point instead. In the former case, this is illustrated by the next remark.
	\begin{remark}
Suppose that $\alpha\colon H_I(X_L;\Z)\to U(1))$ is the colored representation associated to $\omega=(e^{2\pi in _1/q}, \dotsc , e^{2\pi i n_k/q})\in \T_*^n$, as in the statement of Theorem~\ref{theocf}, and suppose that $r$ is a compatible rational framing.
It is easy to see that, for this choice of $\alpha$, a compatible integer framing also exists. We can then apply \ref{theocf} with this integer framing, and then imitate the proof of Theorem~\ref{theorational} using the formula obtained in this way instead of Corollary~\ref{corseifert}. As a result, we will get a formula relating the rho invariant of $S_L(r)$ with the multivariable signature of $L$. However, this is better dealt with on case-by-case basis, as a general formula would be quite cumbersome.
\end{remark}

\begin{remark}
If $L$ is component-to-color algebraically split and $g$ is an integer framing on $L$ which is compatible with the given $\omega \in \T_*^n$, we can apply either Theorem~\ref{theointeger} or Theorem~\ref{theorational}. The fact that the resulting formulas are compatible can be verified by a quick computation using Corollary~\ref{corlensinteger}.
\end{remark}

\bibliographystyle{plain}
\bibliography{biblio2} 
\end{document}